\documentclass[11pt,english]{amsart}
\PassOptionsToPackage{obeyFinal}{todonotes}
\usepackage[T1]{fontenc}
\usepackage[latin9]{inputenc}
\usepackage{color}
\usepackage{babel}
\usepackage{todonotes}
\usepackage{amsbsy}
\usepackage{amstext}
\usepackage{amsthm}
\usepackage{amssymb}
\usepackage{graphicx}
\usepackage{enumitem}
\usepackage[letterpaper]{geometry}
\usepackage[pdfusetitle,
 bookmarks=true,bookmarksnumbered=false,bookmarksopen=false,
 breaklinks=false,pdfborder={0 0 1},backref=false,colorlinks=false]
 {hyperref}
\hypersetup{
 colorlinks=true,linkcolor=teal,citecolor=purple,linktocpage=true}

\makeatletter
\numberwithin{equation}{section}
\numberwithin{figure}{section}
\theoremstyle{plain}
\newtheorem{thm}{\protect\theoremname}[section]
\theoremstyle{definition}
\newtheorem{defn}[thm]{\protect\definitionname}
\theoremstyle{plain}
\newtheorem{lem}[thm]{\protect\lemmaname}
\theoremstyle{remark}
\newtheorem{rem}[thm]{\protect\remarkname}
\theoremstyle{plain}
\newtheorem{cor}[thm]{\protect\corollaryname}
\theoremstyle{remark}
\newtheorem{claim}[thm]{\protect\claimname}
\theoremstyle{plain}
\newtheorem{prop}[thm]{\protect\propositionname}
\theoremstyle{remark}
\newtheorem{notation}[thm]{Notation}

\@ifundefined{date}{}{\date{}}
\usepackage{amsmath, tikz-cd, amssymb, placeins}
\DeclareRobustCommand*\cal{\@fontswitch\relax\mathcal}

\usetikzlibrary{calc}
\usetikzlibrary{decorations.pathmorphing}

\tikzset{curve/.style={settings={#1},to path={(\tikztostart)
    .. controls ($(\tikztostart)!\pv{pos}!(\tikztotarget)!\pv{height}!270:(\tikztotarget)$)
    and ($(\tikztostart)!1-\pv{pos}!(\tikztotarget)!\pv{height}!270:(\tikztotarget)$)
    .. (\tikztotarget)\tikztonodes}},
    settings/.code={\tikzset{quiver/.cd,#1}
        \def\pv##1{\pgfkeysvalueof{/tikz/quiver/##1}}},
    quiver/.cd,pos/.initial=0.35,height/.initial=0}

\tikzset{tail reversed/.code={\pgfsetarrowsstart{tikzcd to}}}
\tikzset{2tail/.code={\pgfsetarrowsstart{Implies[reversed]}}}
\tikzset{2tail reversed/.code={\pgfsetarrowsstart{Implies}}}
\tikzset{no body/.style={/tikz/dash pattern=on 0 off 1mm}}

\makeatother

\providecommand{\claimname}{Claim}
\providecommand{\corollaryname}{Corollary}
\providecommand{\definitionname}{Definition}
\providecommand{\lemmaname}{Lemma}
\providecommand{\propositionname}{Proposition}
\providecommand{\remarkname}{Remark}
\providecommand{\theoremname}{Theorem}

\begin{document}
\title{Knotting and linking in 4 and 5 dimensions from barbell diffeomorphisms}
\author{Seungwon Kim}
\author{Gheehyun Nahm}
\author{Alison Tatsuoka}

\thanks {SK was supported by National Research Foundation of Korea (NRF) grants funded by the Korean government (MSIT) (No.\ 2022R1C1C2004559). GN was partially supported by the ILJU Academy and Culture
Foundation, the Simons collaboration \emph{New structures in low-dimensional
topology}, and a Princeton Centennial Fellowship. AT was partially supported by an NSF Graduate Research Fellowship and by the NSF under Grant No. DMS-1928930
while in residence at the Simons Laufer Mathematical Sciences Institute in Berkeley, California,
during the Spring 2026 semester.}

\begin{abstract}
In this paper, we construct infinitely many non-isotopic 3-knots in the 5-sphere, each of which has four critical points with respect to the standard height function of the 5-sphere. 
This contrasts with a theorem of Scharlemann which says that any 2-knot in the 4-sphere with four critical points is unknotted, and also provides infinitely many knotted solid tori in the 4-sphere and 5-ball, which resolves the last remaining case of the conjecture by Budney and Gabai on the existence of knotted handlebodies.

We also construct various knotted and linked handlebodies, discs, and spheres in the 4-sphere, 5-ball, and 5-sphere, extending recent works of Hughes, Miller, and the first author, and a recent work of the authors. All of our examples are explicit and are constructed using barbell diffeomorphisms.

\end{abstract}

\maketitle
\tableofcontents{}

\section{\label{sec:Introduction}Introduction}

In breakthrough work \cite{budney2021knotted3ballss4}, Budney and Gabai introduced barbell diffeomorphisms and used them to construct infinitely many knotted 3-balls in $S^4$. In the same paper, they conjectured \cite[Conjecture~11.3]{budney2021knotted3ballss4} that there exist knotted handlebodies of any genus $g\ge 1$ in $S^4$.
Hughes, Miller, and the first author \cite{hughes2024knotted} resolved this conjecture for $g\ge 2$ and moreover showed that their handlebodies stay knotted even if their interiors are pushed into $B^5$. 
They first constructed nontrivial $3$-knots in $S^5$ with $(2g+2)$ critical points with respect to the standard height function of $S^5$: one critical point each of index $0$ and $3$, and $g$ critical points each of index $1$ and $2$.
Then they showed that this yields knotted genus $g$ handlebodies in $S^4$.

In this paper, we construct infinitely many nontrivial $3$-knots with four critical points, one critical point each of index $0,1,2,3$.
This contrasts with Scharlemann's theorem \cite{scharlemann1985smoothspheres} which says that a $2$-knot in $S^4$ with four critical points must be unknotted, and also resolves the last remaining case of Budney and Gabai's conjecture, namely the existence of knotted solid tori.\footnote{It will be clear from the construction of our $3$-knots that that they are nontrivial implies Corollary~\ref{cor:knotted-handlebodies}, but in fact Theorem~\ref{thm:nontrivial-3knots}, as stated, also implies Corollary~\ref{cor:knotted-handlebodies}; see Remark~\ref{rem:genus-1-handlebody-from-3knot}.}

\begin{thm}[``Morse-simple'' 3-knots; Theorem~\ref{thm:morsesimple-s3}]
\label{thm:nontrivial-3knots}There exist infinitely many pairwise non-isotopic
embeddings of $S^{3}$ in $S^{5}$ all of which have four critical points with respect
to the standard height function on $S^{5}$ (which is Morse on the $3$-knot).
\end{thm}

\begin{defn}[Compressing-curve equivalent handlebodies]
Let $H_1$ and $H_2$ be two $3$-dimensional handlebodies with common boundary $K$. 
We say that $H_1$ and $H_2$ are \emph{compressing-curve equivalent} if a closed curve $\gamma \subset K$ bounds a disk in $H_1$ if and only if it bounds a disk in $H_2$.
\end{defn}

\begin{cor}[Knotted handlebodies; Theorem~\ref{thm:genus1-handlebody} and Corollary~\ref{cor:bgconj}]
\label{cor:knotted-handlebodies}For each $g\geq1$, there exist infinitely many genus $g$ compressing-curve equivalent handlebodies in $S^4$ which are pairwise non-isotopic rel.~$\partial$ and remain non-isotopic even when their interiors are pushed into $B^{5}$. 
\end{cor}

\begin{rem}[Comparison with \cite{hughes2024knotted}]
The $3$-knot in \cite{hughes2024knotted} was shown to be nontrivial by applying Ruberman's theorem \cite{MR709569} which says that the $5$-twist spun trefoil is not doubly slice, which in turn uses Rokhlin's theorem \cite{MR52101}; they could not explicitly identify either the $3$-knot or the handlebody.
    In contrast, we explicitly construct the $3$-knots and the handlebodies, and we show that our $3$-knots are nontrivial by computing the $\pi_2$ of the complement; this allows us to construct and distinguish infinitely many $3$-knots with few critical points and infinitely many knotted handlebodies.
\end{rem}

Note that the 3-knots of Theorem~\ref{thm:nontrivial-3knots} have the fewest critical points possible while being nontrivial. 
In fact, for all $n\ge 3$, the unknotting theorem for $n$-knots in $S^{n+2}$ \cite{MR179803, MR184249, MR230325} implies that any embedding of $S^n$ in $S^{n+2}$ with two critical points must be unknotted.

In higher dimensions, Ferus \cite[Theorem A.6]{ferus1968totale} (also
see \cite{MR212817}) showed for all $m\ge 1$ that there exist embeddings of (possibly exotic) $(4m+1)$-spheres in $S^{4m+3}$ with four critical points. 
This motivated a question of Kuiper \cite[Section 10, page 390]{kuiper1984geometryintotal}, who asked whether nontrivial $n$-knots with four critical points can also exist in dimensions $n=4m-1, m\geq 1$;\footnote{Kuiper phrases this question in slightly different language. For an embedding $f:M\hookrightarrow\mathbb{R}^{N}$, let $\tau[f]$ denote
the infimum of the number of critical points of $\rho\circ g:M\to\mathbb{R}$
for all embeddings $g:M\hookrightarrow\mathbb{R}^{N}$ isotopic to
$f$ and linear projections $\rho:\mathbb{R}^{N}\to\mathbb{R}$ such
that $\rho\circ g$ is a Morse function. In \cite{kuiper1984geometryintotal},
this is phrased in terms of the total absolute curvature $\tau(f)$
of $f$ \cite[(1.1)]{kuiper1984geometryintotal}: $\tau[f]$ is the
infimum of $\tau(g)$ over all embeddings $g$ isotopic to $f$ \cite[paragraph after (1.5)]{kuiper1984geometryintotal}
(see \cite[Theorem 2.1]{MR950555} or \cite[Lemma 3.16]{ferus1968totale}). Kuiper's question asks whether there exist nontrivial $(4m-1)$-knots $f: S^{4m-1} \hookrightarrow \mathbb{R}^{4m+1}$ with $\tau[f]=4$.}
Theorem~\ref{thm:nontrivial-3knots} answers this question for $m=1$. 
In fact, the construction of these 3-knots extends naturally to the construction of infinitely many $(2k-1)$-knots in $S^{2k+1}$ with 4 critical points for all $k \geq 2$, answering Kuiper's question for all $m \geq 1$ (see Appendix \ref{sec:kuiperappendix}).

Theorem~\ref{thm:nontrivial-3knots} also answers a question posed by Carrara, Ruas, and Saeki \cite[Proposition 5.2]{carrara2001}, who showed that if there exists a 3-knot in $S^{5}$ with four critical
points of indices anything other than $0,1,2,3$, then that 3-knot must be unknotted; they asked if this condition on the indices could be dropped \cite[Remark 5.3]{carrara2001}.
Indeed, Theorem~\ref{thm:nontrivial-3knots} shows it cannot be dropped.

Using similar methods, for all $n\ge 2$, we also construct infinitely many nontrivial $n$-component links of
3\nobreakdash-spheres in $S^{5}$, all of which have $(2n+2)$ critical points with
respect to the standard height function on $S^{5}$. As in the 1-component
case, it follows from a result of Komatsu \cite{MR1364071} (compare \cite{powellspanning})
that any $n$-component link of 3-spheres in $S^{5}$ with $2n$ critical points
must be isotopic to the unlink. Moreover, the links are \emph{Brunnian}, i.e.\ they become isotopic to the $(n-1)$-component unlink after removing any one of the components.

\begin{thm}[Brunnian 3-links; Theorem~\ref{thm:linked-6crit}]
\label{thm:linked-3spheres} For all $n\ge 2$, there exist infinitely many pairwise
non-isotopic $n$-component Brunnian 3-links in $S^{5}$, all of which
have $(2n+2)$ critical points with respect to the standard height function on $S^{5}$.
\end{thm}

Similarly to Corollary~\ref{cor:knotted-handlebodies}, Theorem~\ref{thm:linked-3spheres} give rise to Brunnian links of handlebodies.

\begin{cor}[Brunnian handlebody links; Corollary~\ref{cor:solid-torus-ball-brunnian}]
\label{cor:brunnian-handlebodies}
    For all $n\ge 2$ and $(g_1,\cdots,g_n) \in \mathbb{Z}_{\ge 0} ^n $ with $g_n\ge 1$, there exists an infinite family $H_1 , H_2 , \cdots$ of $n$-component handlebody links of genus $(g_1, \cdots ,g_n)$ in $S^4$ with common boundary such that for all $k \neq \ell$,
    \begin{enumerate}
    \item any proper sublinks of $H_k$ and $H_\ell$ with the same boundary are isotopic rel.~$\partial$; and
    \item the links $H_k$ and $H_\ell$ are not isotopic rel.\ $\partial$ and remain non-isotopic even when their interiors are pushed into $B^5$.
    \end{enumerate}
\end{cor}

Hughes, Miller, and the first author constructed 2-component Brunnian links of handlebodies of genus $g_1\ge 0$ and $g_2\ge 4$ in $S^4$ \cite[Theorem~1.5]{hughes2023nonisotopicsplittingspheressplit}. In \cite{kim2026brunnianlinks3balls4sphere}, we constructed infinitely many $n$-component Brunnian links of 3-balls in $S^4$ for all $n\ge 2$; Niu \cite{niu2026brunnianspanning3disks2unlink} showed this independently for $n=2$. Combining the results of \cite{kim2026brunnianlinks3balls4sphere} with Corollary \ref{cor:brunnian-handlebodies}, we have the existence of infinitely many $n$-component Brunnian links of handlebodies of any genus in $S^4$.

\begin{rem}[Brunnian $2$-disk links]
It is natural to ask whether there exist Brunnian links of 2-disks in $S^4$. 
We show in Theorem~\ref{thm:disks-5dlinked} that there exist infinitely many 2-component Brunnian 2-disk links in $S^4$ that remain non-isotopic in $B^5$. 
However, as a consequence of Gabai's 4-dimensional lightbulb theorem \cite{MR4127900}, there do not exist Brunnian links of 2-disks in $S^4$ with more than 2 components; see Corollary~\ref{cor:no-brunnian-2disk}.
\end{rem}

\subsection{4-dimensional methods}

In this paper, we also distinguish various knotted objects in $S^4$ using $4$-dimensional methods. 
Notably, we obtain infinitely many knotted splitting 3-spheres for a large class of ``links'' $K_{1}\sqcup K_{2}\subset S^{4}$.

\begin{defn}[Splitting spheres]
\label{def:splitting-sphere}Let $K=K_{1}\sqcup K_{2}$ be the union
of disjoint submanifolds $K_{1},K_{2}$ of $S^{n}$ that are \emph{split},
i.e.\ they are contained in disjoint copies of $B^{n}$ in $S^{n}$.
A \textit{splitting sphere} $\Sigma$ for $K$ is an embedded $S^{n-1}\subset S^{n}\setminus K$ such that $K_{1}$ and $K_{2}$ lie
in distinct connected components of $S^{n}\setminus\Sigma$. Two splitting spheres
are \emph{isotopic} if they are isotopic in $S^{n}\setminus K$. 
\end{defn}

\begin{thm}[Knotted splitting spheres]
\label{thm:splitting-spheres}There exist infinitely many pairwise
non-isotopic splitting 3-spheres for split links $K_{1}\sqcup K_{2}\subset S^{4}$,
where: 
\begin{enumerate}
\item \label{enu:splitting-spheres-s1s1}$K_{1},K_{2}\cong S^{1}$ (Theorem~\ref{thm:circle-splittingspheres})
\item \label{enu:splitting-spheres-s1knot}$K_{1}$ is any 
surface in $S^{4}$, $K_{2}\cong S^{1}$ (Theorem~\ref{thm:less-simple})
\item \label{enu:splitting-spheres-knotunknot}$K_{1}$ is any surface in $S^{4}$, $K_{2}$ is an unknotted surface of
genus $g\geq1$ (Theorem~\ref{thm:simple-splitting-spheres})
\end{enumerate}
\end{thm}

When $K_1, K_2$ are unknotted surfaces in $S^4$ of genus $g_1, g_2$, respectively, a pair of non-isotopic splitting 3-spheres for $K_1 \sqcup K_2$ were previously constructed by Hughes, Miller, and the first author \cite{hughes2023nonisotopicsplittingspheressplit} when $(g_1,g_2)\in \{(2,2),(2,3),(3,3)\}$ or $g_1\ge 0$ and $g_2\ge 4$. 
In \cite{tatsuoka2025splittingspheresunlinkeds2s}, the third author constructed infinitely many non-isotopic splitting spheres for the trivial link of two $2$-spheres in $S^4$, which we gave an alternative proof of in \cite{kim2026brunnianlinks3balls4sphere}. 
Theorem~\ref{thm:splitting-spheres}~(\ref{enu:splitting-spheres-knotunknot}) completes the proof that there exist infinitely many splitting spheres for the split link of two unknotted surfaces of any genus in $S^4$.

\begin{rem}[$4$-dimensional methods for knotted handlebodies]
   In Theorems~\ref{thm:simple-knotted-handlebody}~and~\ref{thm:genus1-handlebody}, we give simple, alternative proofs of the existence of infinitely many knotted handlebodies of any genus $g\ge 1$ in $S^4$ using $4$-dimensional methods.
\end{rem}

\begin{rem} [$5$-dimensional method for splitting spheres]
   It is possible to give a $5$-dimensional proof of the existence of infinitely many non-isotopic splitting $3$-spheres in $S^4$ for the split link of unknotted surfaces of genus $g_1 , g_2$ for $g_1 \ge 0$, $g_2 \ge 1$.
   We give an argument in Corollary~\ref{cor:5d-splitting} for a special case, namely the existence of a knotted splitting sphere for $g_1 = 0$, $g_2 = 1$; this is a corollary of Theorem~\ref{thm:linked-3spheres} (Theorem~\ref{thm:linked-6crit}).
\end{rem}

\begin{rem}[Barbells]
    All the knotting and linking in this paper is constructed using barbell diffeomorphisms as introduced by Budney-Gabai in \cite{budney2021knotted3ballss4}. Since \cite{budney2021knotted3ballss4}, barbells have been used by various people to further study mapping class groups of 4-manifolds \cite{bghyperbolic, fernandez2024grasperfamiliesspheress2, lin2026daxinvariantslightbulbs, niu2025mappingclassgroup4dimensional} and to construct interesting knotting and linking in 4-manifolds \cite{ tatsuoka2025splittingspheresunlinkeds2s, niu2026brunnianspanning3disks2unlink, kim2026brunnianlinks3balls4sphere}. The use of barbells allows us to build and distinguish our objects very explicitly.
\end{rem}

\subsection{Organization}
In Section \ref{sec:Barbells,-knotted-disks,}, we review the construction of barbell diffeomorphisms.
In Sections~\ref{sec:Simple-barbell-implantations}~and~\ref{sec:More-complicated-barbells}, we use barbell diffeomorphisms to construct simple examples of various knotted and linked objects and show that they are nontrivial mainly by using $4$-dimensional methods.
In Section~\ref{sec:Codimension--submanifolds}, we use $5$-dimensional methods to show our main theorems. The proofs in Sections~\ref{sec:Simple-barbell-implantations}~and~\ref{sec:More-complicated-barbells} and the proofs in Section~\ref{sec:Codimension--submanifolds} are logically independent and can be read separately.

A more detailed summary of the sections is as follows.
In Section \ref{sec:Simple-barbell-implantations}, we consider a class of barbells that we call \emph{simple}, and we use them to construct infinitely many examples of the following:
(a) splitting spheres for two circles (Theorem~\ref{thm:circle-splittingspheres});
(b) splitting spheres for unknotted surfaces both having genus $\ge 1$ (Theorem \ref{thm:simple-splitting});
(c) two-component Brunnian 2-disk links (Theorem~\ref{thm:disks-5dlinked}); and
(d) knotted handlebodies of genus $g\geq2$ (Theorem~\ref{thm:simple-knotted-handlebody}).

In Section~\ref{sec:More-complicated-barbells}, we use slightly
more complicated barbells to construct the splitting spheres of Theorem~\ref{thm:splitting-spheres}~(\ref{enu:splitting-spheres-s1knot})
(Theorem~\ref{thm:less-simple}) and Theorem~\ref{thm:splitting-spheres}~(\ref{enu:splitting-spheres-knotunknot})
(Theorem~\ref{thm:simple-splitting-spheres}). We also provide a 4-dimensional proof of the existence of infinitely many knotted solid tori in $S^4$ (Theorem~\ref{thm:genus1-handlebody}).

In Subsection~\ref{subsec:3knots3links}, we show that there exist infinitely many 3-knots in $S^5$ with $4$ critical points and $n$-component Brunnian 3-links in $S^5$ with $(2n+2)$ critical points for all $n\ge 2$.
In Subsection~\ref{subsec:general3mfds}, we extend our arguments and show for each 3-manifold $Y$ with Heegaard genus 1 that there exist infinitely many embeddings of $Y$ in $S^5$.
In Subsection~\ref{subsec:handlebodycorollaries}, we collect various corollaries of Subsections~\ref{subsec:3knots3links}~and~\ref{subsec:general3mfds} concerning knotted and linked handlebodies in $S^4$ that remain knotted and linked in $B^5$.

In Subsections~\ref{subsec:3knots3links}~and~\ref{subsec:general3mfds}, we need to study $5$-dimensional handle decompositions of the complements of our $3$-knots and $3$-links in order to show that they are nontrivial; we give more details of this in Appendix~\ref{sec:appendixa}. Finally, we generalize our construction of $3$-knots with $4$ critical points to all odd dimensions in Appendix~\ref{sec:kuiperappendix}.

\subsection{Conventions}
All the manifolds we consider are oriented, and all the diffeomorphisms are orientation preserving. If $X$ is an oriented manifold, then we denote $X$ with the opposite orientation by $\overline{X}$.

We write $A\approx B$ to mean that $A$ and $B$ are isotopic. If $F$ is a set, then we say diffeomorphisms and isotopies are rel.\ $F$ if they fix $F$ pointwise.

We denote by $N(A)$ (resp.\ $\mathring{N}(A)$) a closed (resp.\ open) tubular neighborhood of $A$.

We view $S^4$ as $B^4 \cup S^3 \times [-1,1]\cup \overline{B^4}$, and draw objects in $S^4$ by drawing their intersection with the $S^3\times 0$ time slice of $S^4$. We further view $S^3\times 0$ as $\mathbb{R} ^3\cup \{\infty\}$. Not every object we consider will lie fully in this $S^3 \times 0$ slice; in particular, we often consider 2- and 3-spheres in $S^4$ that intersect $S^3 \times 0$ in equatorial 1- and 2-spheres, respectively. In these cases, we imagine the rest of the 2- and 3-spheres as being capped off by 2- and 3-disks, respectively, in nearby time slices $S^3 \times \pm \epsilon$. We also often draw arcs that link surfaces in $S^4$; it should be understood that these arcs do not intersect the surface, and instead wind around a meridional $S^1$ of the surface in the nearby past and future.

\subsection*{Acknowledgements}
GN thanks Peter Ozsv\'{a}th for his continuous support and helpful discussions. AT thanks Dave Gabai for his invaluable mentorship and advice. We thank 
Ryan Budney, 
Kunal Chawla,
David Gabai, 
Maggie Miller, 
Mark Powell, 
Qiuyu Ren, 
Daniel Ruberman,
and Joshua Wang
for helpful conversations.

\section{\label{sec:Barbells,-knotted-disks,}Barbell diffeomorphisms}

\begin{figure}[h]
\begin{centering}
\includegraphics[scale=0.3]{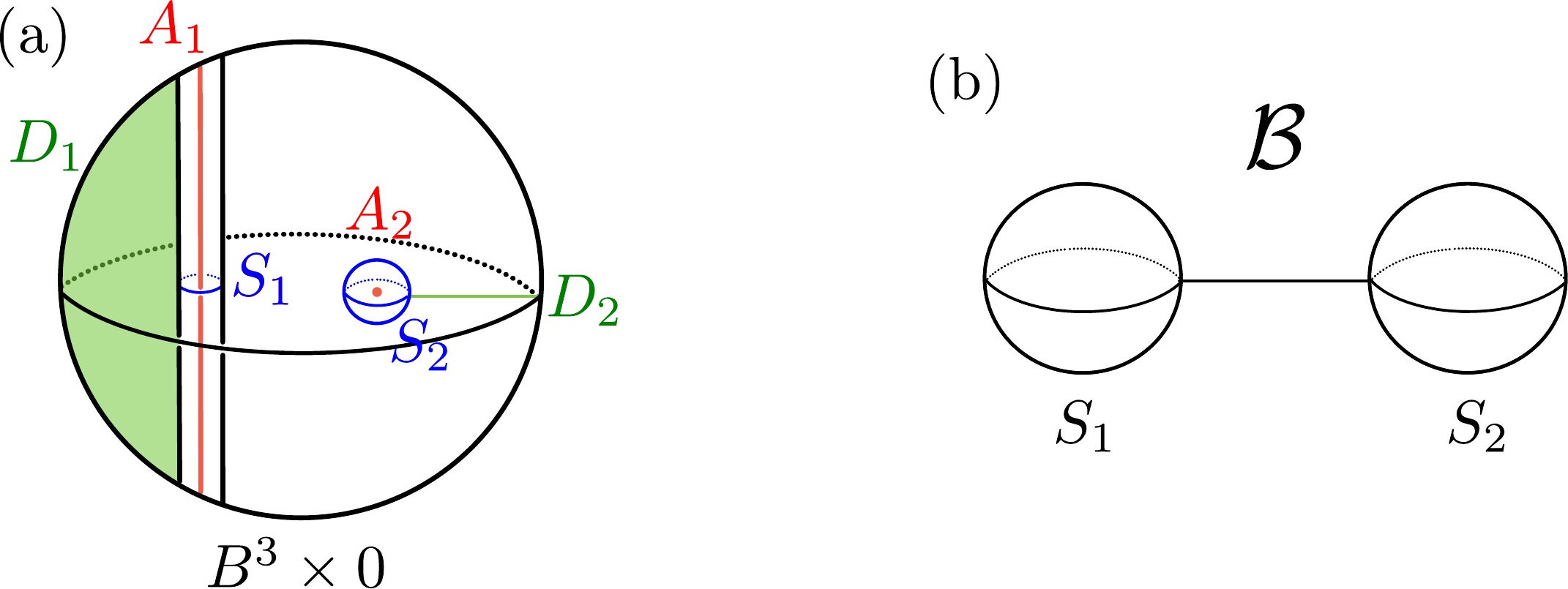}
\par\end{centering}
\caption{\label{fig:barbell}(a): The arcs $A_{1}$ and $A_{2}$, the 2-spheres
$S_{1}$ and $S_{2}$, and the 2-disks $D_{1}$ and $D_{2}$ in
$B^{3}\times0 \subset B^4$. Note that $D_{1}$
appears fully in this 3-dimensional slice, while $D_{2}$ intersects
it in an arc; similarly, $S_{2}$ appears fully in this slice, while
$S_{1}$ intersects it in an equatorial $S^{1}$. (b): The model barbell
$\mathcal{B}$.}
\end{figure}

In \cite{budney2021knotted3ballss4}, Budney and Gabai construct diffeomorphisms of 4-manifolds $X$ up to isotopy rel.~$\partial$ by specifying a \emph{barbell} (Definition~\ref{def:barbell}) in $X$; they call the resulting diffeomorphisms \emph{barbell diffeomorphisms}. 
In this section, we briefly recall how barbell diffeomorphisms are constructed and then review how barbell diffeomorphisms act on surfaces.
We refer to \cite[Section~5]{budney2021knotted3ballss4} or \cite[Section~2]{kim2026brunnianlinks3balls4sphere} for more details.\footnote{Sections 3 and 4 of \cite{kim2026brunnianlinks3balls4sphere} were intended to be sections 6 and 7 of this paper.} 

\begin{defn}[Barbells]
\label{def:barbell}A \emph{barbell} in a $4$-manifold
$X$ is a collection of the following:
\begin{enumerate}
\item two \emph{cuffs}, i.e.\ two disjoint embedded $S^{2}$'s with trivial normal bundle, and
\item a \emph{bar}, i.e.\ an embedded interval $[0,1]$ that connects the cuffs and intersects the cuffs only at its endpoints.
\end{enumerate}
\end{defn}

To define barbell diffeomorphisms, we first define the following objects; see Figure~\ref{fig:barbell}.

\begin{defn}[Model thickened barbell and model barbell]
The \emph{model thinkened barbell} is $\mathcal{NB}:=S^{2}\times D^{2}\natural S^{2}\times D^{2}$. 
Let $S_{1},S_{2}$ be the two copies of $S^{2}\times\mathrm{pt}$ and let $D_{1},D_{2}$ be the two copies of $\mathrm{pt}\times D^{2}$ in $S^{2}\times D^{2}\natural S^{2}\times D^{2}$; then $S_{i}\cap D_{j}=\delta_{ij}$.
The \emph{model barbell} $\mathcal{B}\subset \mathcal{NB}$ is the union of $S_1 , S_2$, and an interval connecting $S_1$ and $S_2$ that intersects $S_i$ only at its endpoints.
\end{defn}

To construct barbell diffeomorphisms, Budney and Gabai first construct a diffeomorphism $\beta: \mathcal{NB} \to \mathcal{NB}$ rel.\ $\partial$, called the \emph{barbell map}. 
Then, if $f$ is an embedding of $\mathcal{NB}$ into a 4-manifold $X$, let $\beta_f \in \mathrm{Diff}_\partial (X)$ be the diffeomorphism obtained by pushing forward $\beta$ along $f$.
Budney and Gabai show that the isotopy rel.\ $\partial$ class of $\beta_f$ only depends on the isotopy class of $f|_{\mathcal{B}}$.
Hence, a barbell $\eta$ in $X$ gives rise to a diffeomorphism which we denote as $\boldsymbol{\eta} \in \mathrm{Diff}_\partial (X)$;
$\boldsymbol{\eta}$ is well-defined up to isotopy rel.\ $\partial$.

Before we study how barbell diffeomorphisms act on surfaces, let us first recall the definition of $\beta$ and how it acts on $D_1$ and $D_2$.
The key observation for defining $\beta $ is that $S^{2}\times D^{2}\natural S^{2}\times D^{2} \cong B^{4}\setminus(\mathring N(A_{1})\sqcup \mathring N(A_{2}))$ where $A_{1}$ and $A_{2}$ are two disjoint, properly embedded arcs in $D^4$;
the barbell map is given by spinning $A_1$ positively around $A_2$.
During the arc-spinning isotopy, the disk $D_{1}$ gets dragged around the
meridional 2-sphere $S_{2}$ to $A_{2}$, so that after the isotopy
$D_{1}$ has changed by a tubing to $S_{2}$ along a path from $A_{1}$
to $A_{2}$. Similarly, the disk $D_{2}$ also gets dragged around
the meridional 2-sphere $S_{1}$ to $A_{1}$, so that $D_{2}$ changes
by a tubing to $\overline{S_{1}}$ along a path from $A_{2}$ to $A_{1}$.

Now we are ready to study how barbell diffeomorphisms $\boldsymbol{\eta}:X\rightarrow X$ act on surfaces $\Sigma\subset X$.

\begin{lem} \label{lem:barbelltubing}
    Let $X$ be a $4$-manifold, and let $\eta$ be a barbell in $X$ with cuffs $S_1, S_2 \subset X$. Let $\Sigma$ be a surface in $X$ such that $\partial \Sigma $ is disjoint from $\eta$. Then in $H_2(X,\partial \Sigma)$ we have
    \begin{equation}\label{eq:barbell-action}
    \boldsymbol{\eta}_\ast ([\Sigma, \partial \Sigma ]) = [\Sigma, \partial \Sigma ] + (\Sigma \cdot S_1 ) [S_2] - (\Sigma \cdot S_2) [S_1 ]
    \end{equation}
    where $\Sigma \cdot S_i$ denotes the algebraic intersection number of $\Sigma$ and $S_i$.
\end{lem}
\begin{proof}
    Let $N(\eta)$ be a small closed neighborhood of $\eta$, and view it as the union of $N(S_1)\sqcup N(S_2)$ and a neighborhood of the bar. Let $D_1 \subset N(S_1)$ (resp.\ $D_2 \subset N(S_2)$) be the normal disk to $S_1$ (resp.\ $S_2$).
    Let us first isotope $\Sigma$ rel.\ $\partial \Sigma$ such that $\Sigma$ is disjoint from the neighborhood of the bar of $\eta$ and that $\Sigma\cap (N(S_1)\sqcup N(S_2))$ consists of parallel copies of some number of $D_1$, $\overline{D_1}$, $D_2$, and $\overline{D_2}$.
    By the above discussion, $\boldsymbol{\eta}(D_1)$ is $D_1$ tubed with $S_2$ along the bar, and $\boldsymbol{\eta}(D_2)$ is $D_2$ tubed with $\overline{S_1}$ along the bar.
    Hence, $\boldsymbol{\eta}(\Sigma)$ can be obtained from $\Sigma$ as follows: for each positive (resp.\ negative) intersection $x\in \Sigma \cap S_1$, tube $\Sigma$ at $x$ with $S_2$ (resp.\ $\overline{S_2}$) along the bar, and for each positive (resp.\ negative) intersection $y\in \Sigma \cap S_2$, tube $\Sigma$ at $y$ with $\overline{S_1}$ (resp.\ $S_1$) along the bar. Thus the lemma follows.
\end{proof}

\begin{rem}\label{rem:budney-gabai-nontrivial-proof}
Note that this is how Budney and Gabai show that $\beta$ is not isotopic to the identity rel.\ $\partial$: 
letting $\beta_{*}$ denote the induced map of $\beta$ on $H_{2}(S^{2}\times D^{2}\natural S^{2}\times D^{2},\partial D_{1})$,
we have 
\begin{equation*}
\beta_{*}[D_{1},\partial D_{1}]=[D_{1},\partial D_{1}]+[S_{2}]\neq [D_{1},\partial D_{1}],
\end{equation*}
and so $\beta$ is a nontrivial diffeomorphism of $S^{2}\times D^{2}\natural S^{2}\times D^{2}$
rel.\ $\partial$.
\end{rem}

\section{\label{sec:Simple-barbell-implantations}Simple barbell implantations}

In this section we consider a special class of barbells, called \emph{simple} barbells (Definition~\ref{def:A-simple-barbell}), and use them to construct various knotted and linked objects in $S^4$. As we will see, the defining property of simple barbells allows us to easily distinguish the isotopy classes of our objects. 

Notably, simple barbells are already strong enough to produce interesting knotted objects in $S^{4}$ and $B^{5}$. For instance, we use them to prove (Theorem~\ref{thm:simple-splitting}) that there exist infinitely
many pairwise non-isotopic splitting spheres for the $2$-component
unlink $K_{m,n}$ of an unknotted genus $m$ and an unknotted genus $n$ surface for all $m,n\ge1$. This extends the main result of \cite{hughes2023nonisotopicsplittingspheressplit}, where Hughes, Miller, and the first author constructed pairs of non-isotopic splitting spheres for $K_{m,n}$ when $m\geq 4$ and $n\ge 0$, or when $(m,n)\in \{(2,2),(2,3),(3,3)\}$. 

We also use simple barbells to give an alternative proof (Theorem~\ref{thm:simple-knotted-handlebody}) of the main result of \cite{hughes2024knotted}, which says that there exist knotted handlebodies of genus $\ge 2$ in $S^4$ that remain non-isotopic when their interiors are pushed into $B^5$. In fact, Theorem \ref{thm:simple-knotted-handlebody} shows that there exist infinitely many knotted handlebodies of genus $\ge 2$ in $S^4$ that are pairwise non-isotopic rel.\ $\partial$; later in Section \ref{sec:Codimension--submanifolds}, we show that these handlebodies also stay pairwise non-isotopic when their interiors are pushed into $B^5$ (Corollary~\ref{cor:simple-5d}).

Now we define simple barbells.

\begin{defn}
\label{def:A-simple-barbell}Let  $X$ be a 4-manifold, and let $F$ be a subset of $X$. A \emph{simple barbell} in $(X,F)$
is a barbell in $X\setminus F$ with cuffs $S_{L}$
and $S_{R}$, such that (1) the homology class of $S_{L}$ is nonzero in $H_2(X,F)$, 
and (2) there exists an embedded disk $(D_{R},\partial D_{R})\subset(X,F)$
that is disjoint
from $S_{L}$ and intersects $S_{R}$ transversely at one point.
\end{defn}

See Figure~\ref{fig:simplebarbell} for the key example of a simple barbell: $F$ is the union of two $S^{1}$'s, and $\beta$ is a barbell in the complement $S^{4}\setminus F$.
Before moving on to the results of this section, we explain our motivation for defining simple barbells.
Recall from Remark~\ref{rem:budney-gabai-nontrivial-proof} that Budney and Gabai show the barbell map $\beta:\cal{NB} \to \cal NB$ is not isotopic to the identity rel.\ $\partial$ by showing the induced map $\beta _\ast $ acts nontrivially on $H_2 ({\cal{NB}},\partial D_{1} )$. We import this to the setting of simple barbells via Equation~(\ref{eq:barbell-action}): if $\eta$ is a simple barbell in $X \setminus F$, then the induced map $\boldsymbol{\eta}_\ast$ on $H_2 (X ,F)$ satisfies
\begin{equation} \label{eq:simple-homology}
\boldsymbol{\eta}_\ast ([D_R , \partial D_R ]) = [D_R , \partial D_R ] - [S_L]\neq [D_R , \partial D_R ] \in H_2 (X ,F).
\end{equation}

In what follows, we use Equation~(\ref{eq:simple-homology}) repeatedly to show that the objects we construct by implanting simple barbells are not isotopic in $S^4$ rel.\ $F$.

\begin{figure}[h]
\begin{centering}
\includegraphics{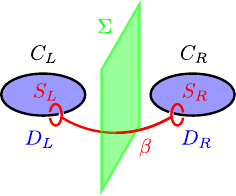}
\par\end{centering}
\caption{\label{fig:simplebarbell}Circles $C_{L}$ and $C_{R}$, disks $D_{L}$
and $D_{R}$, a simple barbell $\beta$ with cuffs $S_{L}$ and $S_{R}$, and a splitting 3-sphere $\Sigma$ in $S^{4}$; their
intersections with the $S^{3}\times0$ time slice are drawn. Note
that only the equatorial $S^{1}$\textquoteright s of the cuffs of
the barbell, and the equatorial $S^{2}$ of the splitting sphere $\Sigma$
are drawn.}
\end{figure}

We begin by considering barbells in the complement of two circles in $S^4$. As the first application of Equation~(\ref{eq:simple-homology}), we show that the link of two circles in $S^4$ admits infinitely many pairwise non-isotopic
splitting 3-spheres. We use the notation
of Figure~\ref{fig:simplebarbell}, and denote the barbell diffeomorphism induced by the barbell $\beta$ as $\boldsymbol{\beta} :(S^4 , C_L \sqcup C_R ) \to (S^4 , C_L \sqcup C_R)$.

\begin{thm}
\label{thm:circle-splittingspheres}Let $\Sigma$ be any splitting
sphere for $C_{L}\sqcup C_{R}$. For integers $k\neq\ell$,
$$\boldsymbol{\beta}^k \Sigma := \overbrace{\boldsymbol{\beta} \circ \cdots \circ \boldsymbol{\beta}}^{k}(\Sigma) \mathrm{\ and\ }\boldsymbol{\beta}^{\ell}\Sigma := \overbrace{\boldsymbol{\beta} \circ \cdots \circ \boldsymbol{\beta}}^{\ell}(\Sigma),$$
are not isotopic in $S^4 \setminus (C_{L}\sqcup C_{R})$.
\end{thm}

\begin{proof}
We show that for $k\neq0$, $\boldsymbol{\beta}^{k}\Sigma$ and
$\Sigma$ are not isotopic in $S^4$ rel.\ $N(C_{L}\sqcup C_{R})$. Let $X:=S^{4}\setminus \mathring N(C_{L}\sqcup C_{R})$,
$L:=S^{4}\setminus \mathring N(C_{L})$ and $R:=S^{4}\setminus \mathring N(C_{R})$. Then,
since $X\cong L\#R$, we have
\[
H_{2}(X,\partial D_{R})\cong H_{2}(L)\oplus H_{2}(R,\partial D_{R}).
\]
Note that $L\cong D^{2}\times S^{2}$, and so $H_{2}(L)\cong\mathbb{Z}$
which is generated by the meridian $S_L$ of $C_{L}$.

Now, by Equation\ (\ref{eq:simple-homology}), we have 
\[
[\boldsymbol{\beta}^{k}D_{R},\partial D_{R}]=-k[S_{L}]+[D_{R},\partial D_{R}]\in H_{2}(X,\partial D_{R}),
\]
but if $\boldsymbol{\beta}^{k}\Sigma$ and $\Sigma$ were isotopic in $X$ rel.\ $\partial X$, then this isotopy would isotope $\boldsymbol{\beta}^{k}D_{R}$
into $R$ rel.\ $\partial D_R$. However, $-k[S_L] \neq 0 \in H_2 (L)$, and so 
\[
[\boldsymbol{\beta}^{k}D_{R},\partial D_{R}]\not\in H_{2}(R,\partial D_{R}),
\]
which is a contradiction.
\end{proof}
The barbell $\beta$ also gives rise to Brunnian links of $2$-disks in $S^{4}$
that stay linked even when their interiors are pushed into $B^{5}$.
\begin{thm}
\label{thm:disks-5dlinked}For $k\neq\ell$, $\boldsymbol{\beta}^{k}(D_{L}\sqcup D_{R})$
and $\boldsymbol{\beta}^{\ell}(D_{L}\sqcup D_{R})$ are not isotopic rel.\ $\partial$, and they stay non-isotopic rel.\ $\partial$ even when their interiors
are pushed into $B^{5}$. On the other hand, $\boldsymbol{\beta}^k(D_L)$ and $D_L$ are isotopic rel.\ $C_L$ and $\boldsymbol{\beta}^{k}(D_R)$ and $D_R$ are isotopic rel.\ $C_R$.
\end{thm}

\begin{proof}
First we show $\boldsymbol{\beta}^k(D_L) \approx D_L$ rel.\ $\partial D_L=C_L$. It is sufficient to show this for $k=1$. By the proof of Lemma~\ref{lem:barbelltubing}, $\boldsymbol{\beta}(D_L)$ is obtained by tubing the disk $D_L$ to the cuff $S_R$ of $\beta$. Since $S_R$ bounds a 3-ball disjoint from $C_L$, we can use this 3-ball to isotope $\boldsymbol{\beta}(D_L)$ back to $D_L$ rel.\ $C_L$. Similarly, $\boldsymbol{\beta}^k(D_R)$ is isotopic to $D_R$ rel.\ $C_R$.

Now we show that the disks are linked. We abuse notation and let $\boldsymbol{\beta}^{k}(D_{L})$, $\boldsymbol{\beta}^{\ell}(D_{L})$,
$\boldsymbol{\beta}^{k}(D_{R})$, and $\boldsymbol{\beta}^{\ell}(D_{R})$ denote both the disks
in $S^{4}$ and their push-ins into $B^{5}$. Consider the disk links
\[
\boldsymbol{\beta}^{k}(D_{L}\sqcup D_R ),\ \boldsymbol{\beta}^{\ell}(D_{L} \sqcup D_R )\subset B^{5}.
\]
Then viewing $S^{5}$ as $B^{5}\cup\overline{B^{5}}$, gluing these
disk links together gives us a two-component link of $2$-spheres
$U_{L}\sqcup U_{R}\subset S^{5}$, where 
\[
U_{L}:=\boldsymbol{\beta}^{k}(D_{L})\cup\overline{\boldsymbol{\beta}^{\ell}(D_{L})},\ U_{R}:=\boldsymbol{\beta}^{k}(D_{R})\cup\overline{\boldsymbol{\beta}^{\ell}(D_{R})}.
\]
Using Equation\ (\ref{eq:simple-homology}), one can show that
\[
[U_{R}]=(\ell-k)\mu_{L}\in H_{2}(S^{5}\setminus U_{L})=\left\langle \mu_{L}\right\rangle \cong\mathbb{Z}
\]
where $\mu_{L}$ is represented by a meridional $S^2$ of $U_{L}$ in $S^5$. On the other hand, if $\boldsymbol{\beta}^{k}(D_{L}\sqcup D_{R})$
and $\boldsymbol{\beta}^{\ell}(D_{L}\sqcup D_{R})$ were isotopic rel.\ $\partial$
in $B^{5}$, then $[U_{R}]$ would be $0$ in $H_{2}(S^{5}\setminus U_L)$,
which is a contradiction.
\end{proof}

We briefly digress from simple barbells to discuss Brunnian 2-disks. In \cite{kim2026brunnianlinks3balls4sphere}, we constructed $n$-component Brunnian links of 3-balls in $S^4$ for all $n\ge 2$. In light of Theorem \ref{thm:disks-5dlinked}, it would be natural to ask whether there also exist $n$-component Brunnian links of 2-disks in $S^4$ for $n>2$.
We show that Gabai's
$4$-dimensional light bulb theorem \cite[Theorem 10.1]{MR4127900} implies that they do not exist for all $n>2$.
\begin{prop}[{\cite[Theorems 10.4 and 10.5]{MR4127900}}]
\label{prop:2disk-homology}Let $C=\sqcup_{k=1}^{n}C_{k}$ be an
embedding of $\sqcup_{k=1}^{n}S^{1}$ in $S^{4}$. Let $D,D'$ be
two embeddings of $\sqcup_{k=1}^{n}D^{2}$ in $S^{4}$, that restrict
to $C$ on the boundary. For $k=1,\cdots,n$, let $D_{k}$ (resp.\ $D_{k}'$)
be the restriction of $D$ (resp.\ $D'$) onto the $k$th $D^{2}$
of $\sqcup^{n}D^{2}$. Then, $D$ and $D'$ are isotopic rel.\ $\partial$
if and only if for all $k=1,\cdots,n$ we have
\[
[D_{k}\cup\overline{D_{k}'}]=0\in H_{2}(S^{4}\setminus \mathring N(C\setminus C_{k}))\cong\mathbb{Z}^{n-1}.
\]
\end{prop}

\begin{proof}
This follows from the same argument as \cite[Theorems 10.4 and 10.5]{MR4127900},
using \cite[Theorem 10.1]{MR4127900}. Compare \cite[Theorem 1.8]{krushkal2024corksexoticdiffeomorphisms},
\cite[Theorem 1.3]{ren2025khovanovskeinlasagnamodules}.
\end{proof}
\begin{cor}
\label{cor:no-brunnian-2disk}There are no Brunnian $2$-disk
links in $S^{4}$ with $3$ or more components.
\end{cor}

\begin{proof}
Let $C:=\bigsqcup_{k=1}^{n}C_{k}$ be an embedding of $n$ copies of $S^{1}$
in $S^{4}$, and let $D:=\bigsqcup_{k=1}^{n}D_{k}$ and $D':=\bigsqcup_{k=1}^{n}D_{k}'$ be embeddings of $\sqcup_{k=1}^{n}D^{2}$ in $S^{4}$ such that $\partial D_{k} = \partial D_{k}'=C_{k}$. 
Let $n\geq3$ and suppose that for all $k=1,\cdots ,n$, $D\setminus D_k $ and $D' \setminus D_k '$ are isotopic rel.\ $C\setminus C_k$. We claim that $D$ and $D'$ are isotopic rel.\ $C$. By Proposition~\ref{prop:2disk-homology}, it is sufficient to show that 
\begin{equation}\label{eq:disks-homclass}
[D_{k}\cup\overline{D_{k}'}]=0\in H_{2}(S^{4}\setminus \mathring N(C\setminus C_{k}))
\end{equation}
for all $k=1,\cdots ,n$. To keep the notations simple, we show it for $k=1$; the other cases follow in the same way.

Let $\mu_{i}$ be the meridian of $C_{i}$. Then, $H_{2}(S^{4}\setminus \mathring N(C\setminus C_{1}))\cong\mathbb{Z}^{n-1}$
has basis $\mu_{2},\cdots,\mu_{n}$, and the map $H_{2}(S^{4}\setminus \mathring N(C\setminus C_{1}))\to H_{2}(S^{4}\setminus \mathring N(C\setminus(C_{1}\sqcup C_{k})))$
induced by the inclusion of spaces is given by quotienting out by $\mu_{k}$. Let $a_2, \cdots ,a_n \in \mathbb{Z}$ be such that 
\[
[D_{1}\cup\overline{D_{1}'}]=\sum_{\ell=2}^{n}a_{\ell}\mu_{\ell}\in H_{2}(S^{4}\setminus \mathring N(C\setminus C_{1})).
\]
Let $k\in \{2,\cdots ,n\}$. Since $D\setminus D_{k}$ is isotopic to $D'\setminus D_{k}'$ rel.\ $C\setminus C_{k}$, we have 
\[
[D_{1}\cup\overline{D_{1}'}]=0\in H_{2}(S^{4}\setminus \mathring N(C\setminus(C_{1}\sqcup C_{k}))),
\]
and so $a_{\ell}=0$ for all $\ell\neq k$. Since $n\ge3$, we can
let $k=2$ or $3$, and so we conclude that $a_{\ell}=0$ for all
$\ell=2,\cdots,n$. Therefore Equation~(\ref{eq:disks-homclass}) follows for $k=1$.
\end{proof}

\begin{figure}[h]
\begin{centering}
\includegraphics{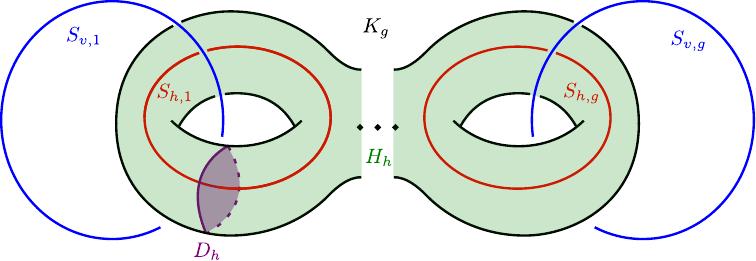}
\par\end{centering}
\caption{\label{fig:kg}The standard genus $g$ surface $K_g \subset S^3 \times 0 \subset S^4 $, the horizontal handlebody $H_h \subset S^3 \times 0 \subset S^4$ that $K_g$ bounds, $2$-spheres $S_{h,i},S_{v,i} \subset S^4 \setminus \mathring {N}(K_g)$ for $i=1,\cdots ,g$
whose homology classes generate $H_{2}(S^{4}\setminus K_{g}) \cong\mathbb{Z}^{2g}$, and a compressing disk $D_h$ for the handlebody $H_h $ that is geometrically dual to $S_{h,1}$.}
\end{figure}

Now we return to simple barbells. In Theorems \ref{thm:simple-splitting} and \ref{thm:simple-knotted-handlebody}, we construct infinitely many pairwise non-isotopic splitting spheres of $K_{m,n}$ for $m,n\ge 1$, where $K_{m,n}$ is the unlink of an unknotted genus $m$ and an unknotted genus $n$ surface in $S^4$, as well as infinitely many knotted handlebodies in $S^4$ of genus $\ge 2$. We do so by implanting simple barbells in the complement of unknotted surfaces with genus $\ge 1$. Below we define the relevant objects in these surface complements (see also Figure~\ref{fig:kg}).

Let $K_{g}\subset S^{3}\times0\subset S^{4}$ be the standard (unknotted)
genus $g$ surface in $S^{4}$. It bounds two genus $g$ handlebodies
in $S^{3}\times0=\mathbb{R}^{3}\cup\{\infty\}$. We call these the
\emph{horizontal handlebody} $H_{h}$ and the \emph{vertical handlebody}
$H_{v}$, which are characterized by the property that the point at
infinity $\infty$ is contained in $H_{v}$.

By Alexander duality, $H_{2}(S^{4}\setminus K_{g})\cong\mathbb{Z}^{2g}$ and is spanned by the homology classes
of the $2g$ many $2$-spheres drawn in Figure~\ref{fig:kg}. We refer to these as the \emph{horizontal spheres}
$S_{h,1},\cdots,S_{h,g}$ and the \emph{vertical spheres} $S_{v,1},\cdots,S_{v,g}$.
Note that for $i\neq j$, we have $S_{h,i}\cap S_{h,j}=S_{v,i}\cap S_{v,j}=S_{h,i}\cap S_{v,j}=\emptyset$,
but $S_{h,i}$ and $S_{v,i}$ intersect (they can be made to intersect
at only two points, with opposite sign). Finally, let $D_{h} \subset H_h$ be the compressing disk of $H_h$ drawn in Figure~\ref{fig:kg}: it intersects $S_{h,1}$ once, and is disjoint from all the other horizontal and vertical 2-spheres. 

\begin{figure}[h]
\begin{centering}
\includegraphics[scale=0.8]{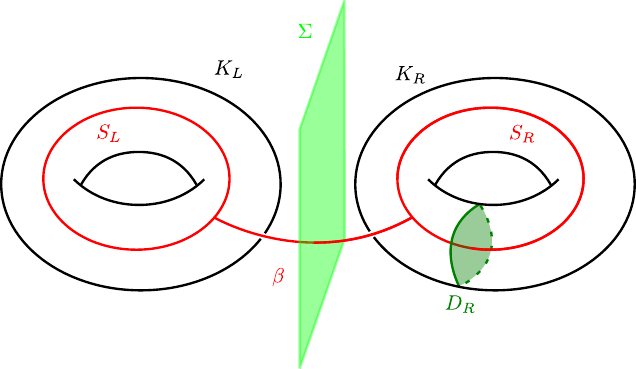}
\par\end{centering}
\caption{\label{fig:sigma11-splitting}The $2$-component unlink of two tori $K_{1,1} = K_{L}\sqcup K_{R}\subset S^{3}\times0$,
a simple barbell $\beta$ with cuffs $S_L$ and $S_R$, a compressing disk $D_{R}$ for the horizontal
handlebody that $K_{R}$ bounds, and a splitting $3$-sphere $\Sigma$.}
\end{figure}

Now we are ready to construct more splitting spheres. Let $K_{m,n}:=K_{L}\sqcup K_{R}$ be the $2$\nobreakdash-component unlink of $K_L$, an
unknotted surface of genus $m$, and $K_R$, an unknotted surface of genus
$n$. Denote the first horizontal sphere for $K_{L}$ (resp.\ $K_{R}$) as $S_L$ (resp.\ $S_R$); see Figure~\ref{fig:sigma11-splitting} for $(n,m)=(1,1)$.
\begin{thm}
\label{thm:simple-splitting}Let $m,n\ge1$, let $\Sigma$ be any
splitting sphere of $K_{m,n}=K_{L}\sqcup K_{R}$, and let $\beta$
be a barbell in $S^{4}\setminus K_{m,n}$ whose two cuffs are $S_L$
and $S_R$. Then, for $k\neq\ell$, $\boldsymbol{\beta}^{k}\Sigma$ and
$\boldsymbol{\beta}^{\ell}\Sigma$ are not isotopic in $S^{4}\setminus K_{m,n}$.
\end{thm}

\begin{proof}
Let $D_{R}$ be the compressing $2$-disk  of Figure~\ref{fig:kg} for the horizontal handlebody
bounded by $K_{R}$ (see Figure~\ref{fig:sigma11-splitting}); in particular, it intersects $S_R$ once and is disjoint from $S_L$. Now, the proof of this theorem
follows almost line-by-line the proof of Theorem~\ref{thm:circle-splittingspheres},
only here we replace $L$ (resp.\ $R$) in the proof of Theorem~\ref{thm:circle-splittingspheres}
by $S^{4}\setminus \mathring N(K_{L})$ (resp.\ $S^{4}\setminus \mathring N(K_{R})$).
\end{proof}

\begin{figure}[h]
\begin{centering}
\includegraphics{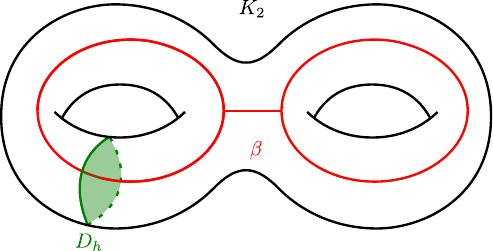}
\par\end{centering}
\caption{\label{fig:sigma2-barbells}The standard genus $2$ surface $K_2 \subset S^3 \times 0 \subset S^4 $,
a barbell $\beta$ that gives rise to knotted handlebodies $\boldsymbol{\beta}^k H_h $, and a compressing disk $D_h$ for the handlebody $H_h$}
\end{figure}

Finally we construct knotted handlebodies of genus $\ge 2$.
\begin{thm}
\label{thm:simple-knotted-handlebody}Let $g\ge2$, and let $\beta$
be a barbell in $S^{4}\setminus K_{g}$ whose two cuffs are $S_{h,1}$
and $S_{h,2}$. Then, for $k\neq\ell$, $\boldsymbol{\beta}^{k}H_{h}$ and $\boldsymbol{\beta}^{\ell}H_{h}$
are not isotopic rel.\ $\partial$.
\end{thm}

\begin{proof}
It is sufficient to prove that $\boldsymbol{\beta}^{k}H_{h}$ and $H_{h}$ are
not isotopic rel.\ $\partial$ for $k\neq0$. The genus $2$ case is drawn in Figure~\ref{fig:sigma2-barbells}.

To use Equation~(\ref{eq:simple-homology}), 
we consider the homology class of the $2$-disk $D_h$: 
we have
\[
[\boldsymbol{\beta}^{k}D_h,\partial D_h]=k[S_{h,2}]+[D_h,\partial D_h]\neq[D_h,\partial D_h]\in H_{2}((S^{4}\setminus K_{g})\cup\partial D_h,\partial D_h).
\]
Since $H_{2}(H_{h})=0$, any two properly embedded disks in $H_{h}$
with the same boundary have the same homology class. However, if $\boldsymbol{\beta}^{k}H_{h}$
and $H_{h}$ were isotopic rel.\ $\partial$, then the isotopy maps
$\boldsymbol{\beta}^{k}D_h$ to a properly embedded disk $D'$ in $H_{h}$ with
$\partial D'=\partial D_h$, but 
\[
[D',\partial D']=[\boldsymbol{\beta}^{k}D_h,\partial D_h]\neq[D_h,\partial D_h]\in H_{2}((S^{4}\setminus K_{g})\cup\partial D_h,\partial D_h),
\]
which is a contradiction.
\end{proof}
\begin{rem}
\label{rem:simple-barbel-d5}
In Corollary~\ref{cor:simple-5d}, we show that for $k\neq\ell$, the handlebodies
$\boldsymbol{\beta}^{k}H_{h}$ and $\boldsymbol{\beta}^{\ell}H_{h}$ actually stay non-isotopic rel.\ $\partial$
even when their interiors are pushed into $B^{5}$. 
\end{rem}

\section{\label{sec:More-complicated-barbells}Virtually simple barbell implantations}

In this section we use slightly more complicated barbells to construct more splitting spheres and handlebodies in $S^4$. In particular, in Theorem \ref{thm:simple-splitting-spheres} we construct infinitely many pairwise non-isotopic splitting spheres for the unlink of two unknotted surfaces of genus $m$ and $n$ for all $m \geq 1, n\geq 0$, and in Theorem \ref{thm:genus1-handlebody} we construct infinitely many knotted solid tori in $S^4$; the latter covers the last remaining case of \cite[Conjecture 11.3]{budney2021knotted3ballss4}. We note that later in Section \ref{sec:Codimension--submanifolds}, we give another construction of knotted solid tori in $S^4$, and show that both families of knotted solid tori remain knotted when their interiors are pushed into $B^5$ (Theorem \ref{thm:morsesimple-s3} and Corollary~\ref{cor:bgconj}). We chose to include the construction of Theorem \ref{thm:genus1-handlebody} to further illustrate the techniques of this section.

The barbells used in this section are just slightly more complicated than the simple barbells of Section \ref{sec:Simple-barbell-implantations} (see Definition \ref{def:A-simple-barbell}). For a given subset $F\subset S^4$, the use of simple barbells in $S^4\setminus F$ required $S^4 \setminus F$ to contain two 2-spheres whose homology classes are linearly independent in $H_2(S^4 , F)$.
In this section, we only require that $S^4 \setminus F$ has one 2-sphere which is nontrivial in $H_2(S^4 , F)$, but we also require that $S^4 \setminus F$ has nontrivial $\pi_1$. The upshot is that the barbells in this section will lift to simple barbells in some cover, where we can once again leverage the methods of Section \ref{sec:Simple-barbell-implantations}.

We begin by constructing infinitely many pairwise non-isotopic splitting spheres for the split ``link'' $K \sqcup C \subset S^4$ of a surface $K$ and a circle $C$ in $S^4$.
Let $S,S'$ be two parallel copies of the meridional 2-sphere of $C$, and let $\beta_{k}$ be a barbell in $S^4 \setminus (K\sqcup C)$ with cuffs $S,S'$ and a bar that winds $k$ times around the
meridional circle $\mu_{K}$ of $K$; see Figure~\ref{fig:simplebarbell-1}~(a). 

\begin{figure}[h]
\begin{centering}
\includegraphics{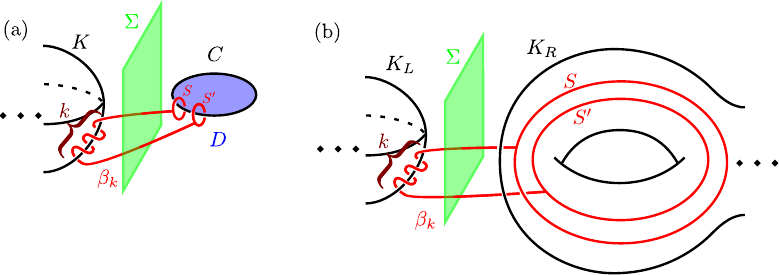}
\par\end{centering}
\caption{\label{fig:simplebarbell-1}(a): A split link $K\sqcup C \subset S^{4}$ of a knotted surface $K$ and
a circle $C$, a splitting sphere $\Sigma$ for $K\sqcup C $, and a barbell
$\beta_{k}$ in the complement of $K\sqcup C\subset S^{4}$. The
cuffs of $\beta_k$ are $S,S'$ the bar winds $k$ times around the meridian
of $K$. (b) A knotted
surface $K_{L}\subset S^{4}$, an unknotted surface $K_{R}$,
a standard splitting sphere $\Sigma$, and a barbell $\beta_{k}$
in the complement of $K_{L}\sqcup K_{R}\subset S^{4}$. The cuffs of $\beta_k$ are
$S,S'$ and the bar winds $k$ times around the meridian of $K_{L}$.}
\end{figure}

\begin{thm}
\label{thm:less-simple}Let $\Sigma$ be any splitting 3-sphere of
$K\sqcup C$. The splitting spheres $\Sigma,\boldsymbol{\beta_1}\Sigma,\boldsymbol{\beta_2}\Sigma,\cdots$
are pairwise non-isotopic rel.\ $K\sqcup C$.
\end{thm}

\begin{proof}
Let $L:=S^{4}\setminus \mathring N(K)$, $R:=S^{4}\setminus \mathring N(C)$, and $X:=S^{4}\setminus \mathring N(K\sqcup C)=L\#R$.
For notational simplicity, we first show that for any $k\ge1$, $\boldsymbol{\beta_{k}}\Sigma$
and $\Sigma$ are not isotopic in $L\#R$ rel.\ $\partial$.

For $m>2k+100$, consider the $m$-fold cyclic
cover $\widetilde{X}$ of $X$ that ``unwinds the meridian $\mu_{K}$ of $K$'',
i.e.\ that corresponds to the kernel of 
\[
\pi_{1}(X)\to H_{1}(X)=\left\langle \mu_{K}\right\rangle \cong\mathbb{Z}\twoheadrightarrow\mathbb{Z}/m\mathbb{Z}.
\]
Similarly let $\widetilde{L}$ be the $m$-fold cyclic cover of $L$ corresponding to the kernel of 
\[
\pi_{1}(L)\to H_{1}(L)=\left\langle \mu_{K}\right\rangle \cong\mathbb{Z}\twoheadrightarrow\mathbb{Z}/m\mathbb{Z}.
\]
Then $\widetilde{X}\cong\widetilde{L}\#\left(\#^{m}R\right)$. Let $\rho:\widetilde{X}\to\widetilde{X}$ be the deck transformation
that corresponds to $\mu_{K}$. Choose a summand  $\widetilde{R}$ of $\#^{m}R$, and let $\widetilde{D},\widetilde{S},$
and $\widetilde{S'}$ denote the lifts of $D,S,$ and $S'$, respectively,
that live inside $\widetilde{R}$.

Then, $H_{2}(\widetilde{X},\partial\widetilde{D})$ decomposes as
\begin{equation}
H_{2}(\widetilde{L})\oplus H_{2}(\widetilde{R},\partial\widetilde{D})\oplus\left(\bigoplus_{i=1}^{m-1}H_{2}(\rho^{i}\widetilde{R})\right).\label{eq:nfold-directsum}
\end{equation}
Let $\widetilde{\boldsymbol{\beta_{k}}}:\widetilde{X}\to\widetilde{X}$ be the
lift of $\boldsymbol{\beta_{k}}$ that fixes $\partial\widetilde{D}$. That $\boldsymbol{\beta_{k}}\Sigma$
and $\Sigma$ are not isotopic in $L\#R$ rel.~$\partial$ follows
from the following two statements:
\begin{enumerate}
\item We have 
\[
[\widetilde{\boldsymbol{\beta_{k}}}\widetilde{D},\partial\widetilde{D}]=[\widetilde{D},\partial\widetilde{D}]+[\rho^{k}\widetilde{S}]-[\rho^{-k}\widetilde{S'}]\in H_{2}(\widetilde{X},\partial\widetilde{D}).
\]
Note that the terms on the right hand side are all nonzero and lie
in different summands of Equation (\ref{eq:nfold-directsum}).
\item If $\boldsymbol{\beta_{k}}\Sigma$ and $\Sigma$ are isotopic in $L\#R$ rel.\ $\partial$,
then the lift of this isotopy to $\widetilde{X}$ takes $\widetilde{\boldsymbol{\beta_{k}}}\widetilde{D}$
into $\widetilde{R}$, and so $[\widetilde{\boldsymbol{\beta_{k}}}\widetilde{D},\partial\widetilde{D}]$
lies in the $H_{2}(\widetilde{R},\partial\widetilde{D})$ summand
of Equation (\ref{eq:nfold-directsum}).
\end{enumerate}
It remains to prove the first statement. The diffeomorphism $\widetilde{\boldsymbol{\beta_{k}}}:\widetilde{X}\rightarrow\widetilde{X}$
is a composition of $m$ barbell diffeomorphisms corresponding to the $m$ barbells in $\widetilde{X}$ that cover the barbell $\beta_{k}$ in $X$. There are two lifts of $\beta_k$ that
intersect the disk $\widetilde{D}$: one barbell whose cuffs are $\widetilde{S}$
and $\rho^{k}\widetilde{S'}$, and one barbell whose cuffs are $\rho^{-k}\widetilde{S}$
and $\widetilde{S'}$. Therefore
the first statement follows by Equation~(\ref{eq:barbell-action}).

The proof that $\boldsymbol{\beta_{\ell}}^{-1}\boldsymbol{\beta_{k}}\Sigma$ and $\Sigma$ are
not isotopic in $L\#R$ rel.\ $\partial$ for positive integers $k\neq\ell$
is similar. Let $m>2k+2\ell+100$, and proceed as before. Then, we have 
\[
[\widetilde{\boldsymbol{\beta_{\ell}}}^{-1}\widetilde{\boldsymbol{\beta_{k}}}\widetilde{D},\partial\widetilde{D}]=[\widetilde{D},\partial\widetilde{D}]+[\rho^{k}\widetilde{S}]-[\rho^{-k}\widetilde{S'}]-[\rho^{\ell}\widetilde{S}]+[\rho^{-\ell}\widetilde{S'}]\in H_{2}(\widetilde{X},\partial\widetilde{D}),
\]
but if $\boldsymbol{\beta_{\ell}}^{-1}\boldsymbol{\beta_{k}}\Sigma$ and $\Sigma$ were isotopic
in $L\#R$ rel.\ $\partial$, then $[\widetilde{\boldsymbol{\beta_{\ell}}^{-1}}\widetilde{\boldsymbol{\beta_{k}}}\widetilde{D},\partial\widetilde{D}]$
would lie in the $H_{2}(\widetilde{R},\partial\widetilde{D})$ summand
of Equation (\ref{eq:nfold-directsum}), which is a contradiction.
\end{proof}
\begin{rem}
The above proof also shows that for any two distinct integer sequences
$(a_{1},\cdots,a_{n})$ and $(b_{1},\cdots,b_{n})$, $\boldsymbol{\beta_{1}}^{a_{1}}\cdots\boldsymbol{\beta_{n}}^{a_{n}}\Sigma$
and $\boldsymbol{\beta_{1}}^{b_{1}}\cdots\boldsymbol{\beta_{n}}^{b_{n}}\Sigma$ are not isotopic
rel.\ $K\sqcup C$.
\end{rem}

The same argument gives infinitely many splitting spheres for any
split link $K_{L}\sqcup K_{R} \subset S^{4}$, where $K_{L}$ is any surface and $K_{R}$ is any
unknotted surface of genus $\ge1$. Let $\beta_{k}$ be a barbell whose cuffs are parallel
copies of the sphere $S_{h,1}$ for $K_{R}$ and whose
bar winds $k$ times around the meridian $\mu_{K_{L}}$ of $K_{L}$; see Figure~\ref{fig:simplebarbell-1}~(b).
\begin{thm}
\label{thm:simple-splitting-spheres}The splitting spheres $\Sigma,\boldsymbol{\beta_{1}}\Sigma,\boldsymbol{\beta_{2}}\Sigma,\cdots$
are pairwise non-isotopic rel.\ $K_{L}\sqcup K_{R}$.
\end{thm}

\begin{proof}
\noindent The same proof as Theorem~\ref{thm:less-simple} works,
replacing $K,C$ by $K_{L},K_{R}$, respectively. The only
difference is that $H_{1}(X)=\left\langle \mu_{K_{L}},\mu_{K_{R}}\right\rangle \cong\mathbb{Z}^{2}$,
where $\mu_{K_{L}}$ (resp.\ $\mu_{K_{R}}$) is the meridian of $K_{L}$
(resp.\ $K_{R}$). The $m$-fold cyclic cover of $X$ now corresponds
to the kernel of 
\[
\pi_{1}(X)\to H_{1}(X)\twoheadrightarrow\mathbb{Z}/m\mathbb{Z},
\]
where the last map is given by quotienting out by $m\mu_{K_{L}}$
and $\mu_{K_{R}}$.
\end{proof}

\begin{figure}[h]
\begin{centering}
\includegraphics{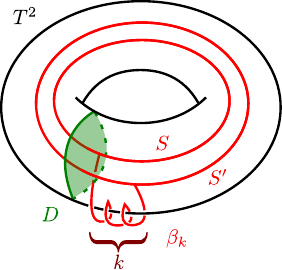}
\par\end{centering}
\caption{\label{fig:genus1-handlebody}The standard torus $T^{2}\subset S^{3}\times0 \subset S^4$, a barbell $\beta_k$ whose cuffs $S,S'$ are parallel copies of the horizontal sphere $S_h$ and its bar winds $k$ times around the meridian of $T^2$, and 
a compressing disk $D$ of the horizontal handlebody $H_h$.}
\end{figure}

Lastly, we construct infinitely many knotted genus $1$ handlebodies in $S^4$. The setting is depicted in Figure \ref{fig:genus1-handlebody}. Let $T^2$ be the standard genus $1$ surface in $S^3\times 0\subset S^4$, let $D$ be a compressing disk of the horizontal handlebody $H_h$, and let the \emph{horizontal barbell} $\beta_k$ be the barbell whose cuffs are two parallel copies $S, S'$ of the horizontal sphere $S_h$, and whose bar winds $k$ times around the meridian of $T^2$. 

\begin{thm}\label{thm:genus1-handlebody}
The genus $1$ handlebodies $\{H_h\}\cup\{\boldsymbol{\beta_{k}}H_h\}_{k\ge1}$
in $S^{4}$ are pairwise non-isotopic rel.~$\partial$.
\end{thm}

\begin{proof}
For simplicity, we first show that $H_h$ and $\boldsymbol{\beta_{k}}H_h$
are non-isotopic rel.\ $\partial$ for $k\ge1$. For $m>2k+100$,
let $\widetilde{X}$ be the $m$-fold cyclic branched cover of $S^{4}$
along $T^{2}$, and let $\rho:\widetilde{X}\to\widetilde{X}$ be the
deck transformation corresponding to the meridian of $T^{2}$.
Let $\widetilde{D}$, $\widetilde{S}$, and $\widetilde{S'}$ be lifts of $D$, $S$, and $S'$, respectively, to $\widetilde{X}$ such that
\begin{equation}
\left|\widetilde{D}\cap\rho^{i}\widetilde{S}\right|=\left|\widetilde{D}\cap\rho^{i}\widetilde{S'}\right|=\begin{cases}
1 & {\rm if\ } i \equiv 0 \mod m\\
0 & {\rm otherwise}
\end{cases}.\label{eq:intersection-bc}
\end{equation}
Let $\widetilde{\boldsymbol{\beta_{k}}}:\widetilde{X}\to\widetilde{X}$ be the
lift of $\boldsymbol{\beta_{k}}$ that fixes $\widetilde{S}$. Then, arguing similarly
to the proof of Theorem~\ref{thm:less-simple}, we have\footnote{Note that there is no particular reason for working over $\mathbb{F}_{2}$
other than to avoid dealing with signs.}
\[
[\widetilde{\boldsymbol{\beta_{k}}}\widetilde{D},\partial\widetilde{D}]=[\widetilde{D},\partial\widetilde{D}]+[\rho^{k}\widetilde{S}]+[\rho^{-k}\widetilde{S'}]\in H_{2}(\widetilde{X},\partial\widetilde{D};\mathbb{F}_{2}).
\]

Assume that there exists an isotopy rel.\ $\partial$ between $\boldsymbol{\beta_{k}}H_h$
and $H_h$. Then, this isotopy would send $\boldsymbol{\beta_{k}}D$ to a compressing
disk of $H_h$. Since any two compressing disks of $H_h$ are
isotopic rel.\ $\partial$, we have that $\boldsymbol{\beta_{k}}D$ and $D$ are isotopic
rel.\ $\partial$. Lifting this isotopy to $\widetilde{X}$, we have
that $\widetilde{\boldsymbol{\beta_{k}}}\widetilde{D}$ and $\rho^{i}\widetilde{D}$
are isotopic rel.\ $\partial$ for some $i$, and so 
\begin{equation}
[\rho^{i}\widetilde{D},\partial\widetilde{D}]=[\widetilde{D},\partial\widetilde{D}]+[\rho^{k}\widetilde{S}]+[\rho^{-k}\widetilde{S'}]\in H_{2}(\widetilde{X},\partial\widetilde{D};\mathbb{F}_{2}).\label{eq:branched-cover-homology}
\end{equation}
We will show that Equation~(\ref{eq:branched-cover-homology}) is
a contradiction, by using the intersection product
\[
H_{2}(\widetilde{X}\setminus \mathring N(\partial\widetilde{D}),\partial N(\partial\widetilde{D});\mathbb{F}_{2})\otimes H_{2}(\widetilde{X}\setminus \mathring N(\partial\widetilde{D});\mathbb{F}_{2})\to\mathbb{F}_{2}.
\]
Note that $H_{2}(\widetilde{X},\partial\widetilde{D};\mathbb{F}_{2})\cong H_{2}(\widetilde{X}\setminus \mathring N(\partial\widetilde{D}),\partial N(\partial\widetilde{D});\mathbb{F}_{2})$.

First, note that $[\rho^{i}\widetilde{S}]=[\rho^{i}\widetilde{S'}]$
and that the intersection product  $[\rho^{i}\widetilde{S}]\cdot[\rho^{j}\widetilde{S}]=0$
for all $i,j$. Hence, taking the intersection product of both sides
of Equation~(\ref{eq:branched-cover-homology}) with $[\widetilde{S}]$,
we have $i=0$ modulo $m$ by Equation~(\ref{eq:intersection-bc}), and so Equation~(\ref{eq:branched-cover-homology}) implies 
\begin{equation}
[\rho^{k}\widetilde{S}]+[\rho^{-k}\widetilde{S'}]=0\in H_{2}(\widetilde{X},\partial\widetilde{D};\mathbb{F}_{2}).\label{eq:bc2}
\end{equation}
Using the long exact sequence for the pair $(\widetilde{X}\setminus \mathring N(\partial\widetilde{D}),\partial N(\partial\widetilde{D}))$,
we see that the kernel of
\[
H_{2}(\widetilde{X}\setminus \mathring N(\partial\widetilde{D});\mathbb{F}_{2})\to H_{2}(\widetilde{X}\setminus \mathring N(\partial\widetilde{D}),\partial N(\partial\widetilde{D});\mathbb{F}_{2})
\]
is spanned by the meridian $\mu$
of $\partial\widetilde{D}$. Hence, Equation~(\ref{eq:bc2}) implies
\begin{equation}
[\rho^{k}\widetilde{S}]+[\rho^{-k}\widetilde{S'}]=0\ {\rm or}\ [\mu]\in H_{2}(\widetilde{X}\setminus \mathring N(\partial\widetilde{D});\mathbb{F}_{2}).\label{eq:bc3}
\end{equation}

Now we arrive at a contradiction by taking the intersection product of both sides of Equation~(\ref{eq:bc3}) with $[\rho^{k}\widetilde{D},\partial\widetilde{D}]$
and $[\widetilde{D},\partial\widetilde{D}]$:
$$([\rho^k\widetilde{S}]+[\rho^{-k}\widetilde{S}'])\cdot [\rho^k\widetilde D, \partial \widetilde{D}]=1$$ by Equation~(\ref{eq:intersection-bc}), so that $[\rho^k\widetilde{S}]+[\rho^{-k}\widetilde{S}']\neq 0$, while  $$([\rho^k\widetilde{S}]+[\rho^{-k}\widetilde{S}'])\cdot [\widetilde
D, \partial \widetilde{D}]=0 \mathrm{\ and\ } [\mu]\cdot [\widetilde
D, \partial \widetilde{D}]=1,$$ so that $[\rho^k\widetilde{S}]+[\rho^{-k}\widetilde{S}']\neq [\mu]$.

Finally, to show that $\boldsymbol{\beta_{k}}H_h$ and $\boldsymbol{\beta_{\ell}}H_h$ are
non-isotopic rel.\ $\partial$ for positive integers $k,\ell$, we
run the same argument as above, working in the $m$-fold cyclic branched cover of $S^4$ along $T^2$ for $m>2k+2\ell+100$. As before, it reduces
to showing
\[
[\rho^{k}\widetilde{S}]+[\rho^{-k}\widetilde{S'}]+[\rho^{\ell}\widetilde{S}]+[\rho^{-\ell}\widetilde{S'}]\notin\left\langle [\mu]\right\rangle \subset H_{2}(\widetilde{X}\setminus \mathring N(\partial\widetilde{D});\mathbb{F}_{2}),
\]
which follows from considering the intersection product of $[\rho^{k}\widetilde{S}]+[\rho^{-k}\widetilde{S'}]+[\rho^{\ell}\widetilde{S}]+[\rho^{-\ell}\widetilde{S'}]$ with
$[\rho^{k}\widetilde{D},\partial\widetilde{D}]$ and $[\widetilde{D},\partial\widetilde{D}]$.
\end{proof}

\section{\label{sec:Codimension--submanifolds} Morse-simple 3-manifolds in $S^{5}$ and handlebodies in $B^5$}

In this section we use barbells to construct knotting and linking in $S^5$ and $B^5$. 

In Subsection \ref{subsec:3knots3links}, we construct nontrivial 3-knots with four critical points (Theorem \ref{thm:morsesimple-s3}) and $n$-component Brunnian 3-links with $2n+2$ critical points (Theorem \ref{thm:linked-6crit}) for all $n\ge 2$ in $S^5$. Theorem \ref{thm:morsesimple-s3} partially answers a question of Kuiper~\cite[Section 10, page 390]{kuiper1984geometryintotal}. It follows from a result of Komatsu \cite{MR1364071} (see also the unknotting theorem for $3$-knots \cite[Corollary 3.1]{MR184249} and \cite[Theorem 2.1]{MR230325}, and also \cite{hartman_unknotting3balls,powellspanning}) that for all $n\ge 1$, any $n$-component link of 3-spheres in $S^5$ with $2n$ critical points is isotopic to the unlink; hence the term \emph{``Morse-simple''}.

In Subsection \ref{subsec:general3mfds}, we generalize the argument of Theorem \ref{thm:morsesimple-s3} to produce infinitely many non-isotopic embeddings of any 3-manifold of Heegaard genus 1 in $S^5$ with four critical points (Theorem \ref{thm:genus1-hd} and Corollary \ref{cor:morsesimple3mfd}).

Finally, in Subsection \ref{subsec:handlebodycorollaries}, we use the results of Subsections \ref{subsec:3knots3links} and \ref{subsec:general3mfds} to obtain knotted handlebodies of all genus $\ge 1$ in $S^4$ that remain knotted when their interiors are pushed into $B^5$ (Corollaries \ref{cor:bgconj}~and~\ref{cor:simple-5d}), as well as $n$-component Brunnian links of handlebodies with one component having genus $\ge 1$ in $S^4$ that remain linked when their interiors are pushed into $B^5$ (Corollary~\ref{cor:solid-torus-ball-brunnian}). The genus 1 knotted handlebodies of Corollary \ref{cor:bgconj} resolve a conjecture of Budney-Gabai \cite[Conjecture 11.3]{budney2021knotted3ballss4} on the existence of knotted handlebodies in $S^4$. The linked handlebodies of Corollary~\ref{cor:simple-5d}, together with the $n$-component Brunnian links of 3-balls constructed in \cite{kim2026brunnianlinks3balls4sphere} and independently by Niu for the $n=2$ case in \cite{niu2026brunnianspanning3disks2unlink}, provide infinitely many $n$-component Brunnian links of handlebodies of any genera in $S^4$. 

Throughout this section, we view $S^5$ as $B^5 \cup \overline{B^5}$ and think of the common boundary of the 5-balls as the equatorial $S^4$ of $S^5$.
We will consider handlebodies $H \subset S^4$ and push their interiors into $B^5$; the resulting pushed-in handlebody will also be denoted by $H\subset B^5$ by abuse of notation.
We construct $3$-manifolds $Y\subset S^5$ by taking two handlebodies $H_1 , H_2 \subset S^4$ with common boundary and taking the union of the push-ins $H_1 \subset B^5$ and $\overline{H_2} \subset \overline{B^5}$; i.e.\ $Y:=H_1 \cup \overline{H_2}$.

Thinking of $B^5$ as the southern hemisphere and $\overline{B^5}$ as the northern hemisphere of $S^5$, then with respect to the standard height function on $S^5$, the handlebody $H_1 \subset B^5$ contributes index-$0$ and index-$1$ critical points to $Y$, while the handlebody $\overline{H_2} \subset \overline{B^5}$ contributes index-$2$ and index-$3$ critical points to $Y$. In our figures, we draw objects in the equatorial $S^4 \subset S^5$ by drawing their intersections with the equatorial $S^3\times 0$ in $S^4$.

\subsection{\label{subsec:3knots3links} Morse-simple knots and links in $S^5$}

We begin this subsection by constructing infinitely many 3-knots in $S^5$ with four critical points. 

\begin{figure}[h]
\begin{centering}
\includegraphics[scale=0.95]{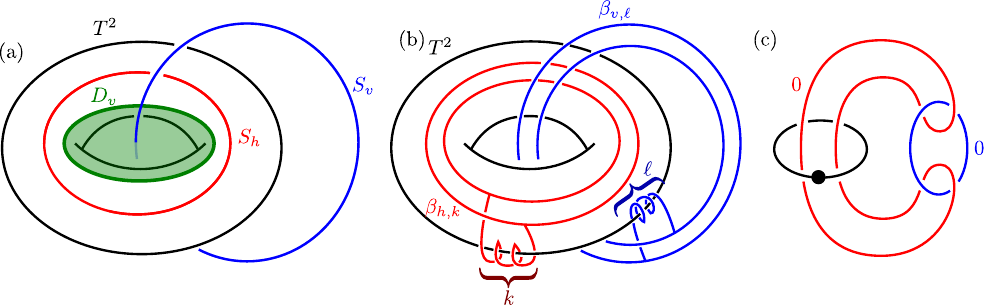}
\par\end{centering}
\caption{\label{fig:tmp2}(a) An unknotted torus $T^{2}\subset S^{3}\times0$,
a compressing disk $D_{v}$ of the vertical handlebody $H_{v}$,
and two $2$-spheres $S_{h}$ and $S_{v}$ whose homology classes
span $H_{2}(S^{4}\setminus T^{2})\protect\cong\mathbb{Z}^{2}$. (b)
The horizontal barbell $\beta_{h,k}$ and the vertical barbell $\beta_{v,\ell}$
in $S^{4}\setminus T^{2}$. (c): A handle diagram for $S^{4}\setminus \mathring N(T^{2})$.
Note that the two $0$-framed $2$-handles correspond to $S_{h}$
and $S_{v}$.}
\end{figure}

For positive integers $k,\ell\in\mathbb{Z}_{>0}$, we consider two
barbells in $S^{4}\setminus T^{2}$ (see Figure~\ref{fig:tmp2}~(b)):
the \emph{horizontal barbell} $\beta_{h,k}$ and the \emph{vertical
barbell} $\beta_{v,\ell}$. The two cuffs of the horizontal barbell
$\beta_{h,k}$ are two parallel copies of $S_{h}$, and the bar
of $\beta_{h,k}$ winds $k$ times around the meridian of $T^{2}$.
Similarly, the two cuffs of the vertical barbell $\beta_{v,\ell}$
are two parallel copies of $S_{v}$, and the bar of $\beta_{v,\ell}$
winds $\ell$ times around the meridian of $T^{2}$.
\begin{thm}
\label{thm:morsesimple-s3}Let $k,\ell\ge1$, and let $Y_{\beta}$
be the $3$-knot $H_{v}\cup\overline{\boldsymbol{\beta_{v,\ell}\beta_{h,k}}H_{h}}$.
Then, 
\[
\dim_{\mathbb{F}_{2}}\left(\pi_{2}(S^{5}\setminus Y_{\beta})\otimes_{\mathbb{Z}}\mathbb{F}_{2}\right)=2k+2\ell+2.
\]
In particular, by varying $k,\ell\ge1$, we obtain infinitely many
pairwise non-isotopic $3$-knots in $S^{5}$ with $4$ critical points
with respect to the standard height function of $S^{5}$.
\end{thm}

\begin{proof}
To simplify notation, we henceforth drop the indices
$k$ and $\ell$ and write $\beta_{h},\beta_{v}$ for the horizontal and vertical barbells $\beta_{h,k}, \beta_{v,\ell}$.

Let $X:=S^{5}\setminus \mathring N(Y_{\beta})$, and let $\widetilde{X}$ be
the universal cover of $X$. By Hurewicz and the universal coefficient
theorem, 
\[
\pi_{2}(X)\otimes_{\mathbb{Z}}\mathbb{F}_{2}\cong H_{2}(\widetilde{X};\mathbb{F}_{2}).
\]

We compute $H_{2}(\widetilde{X};\mathbb{F}_{2})$ from an explicit
handle diagram description of $X$. 

\begin{claim}
\label{claim:pi2-comp}Let $Z:=S^{4}\setminus \mathring N(T^{2})$, and let $\mu$ denote the meridian of $T^2$ in $S^4$, which generates $\pi_1(S^4 \setminus T^2)\cong \mathbb{Z}$. Let $\widetilde{Z}$ be the universal cover of $Z$, and let $\rho: \widetilde{Z}\to \widetilde{Z}$ be the deck transformation corresponding to the meridian $\mu$. Let $\widetilde{D_v}$ (resp.\ $\widetilde{S_{\beta}}$) be a lift of $D_{v}$ (resp.\ $\boldsymbol{\beta_{v}\beta_{h}}S_{v}$)
to $\widetilde{Z}$, and let $\widetilde{S_{\beta}}\cdot\rho^{i}\widetilde{D_{v}}$ denote the mod 2 intersection number of $\widetilde{S_{\beta}}$ and
$\rho^{i}\widetilde{D_{v}}$ in $\widetilde{Z}$. Then 
\begin{equation}
H_{2}(\widetilde{X};\mathbb{F}_{2})\cong\mathbb{F}_{2}[t,t^{-1}]/(f),\label{eq:pi2tensorf2}
\end{equation}
where $f$ is the \textit{mod 2 equivariant intersection number of $\boldsymbol{\beta_v\beta_h}S_v$ and $D_v$}, i.e.\ 
\[
f:=\sum_{i\in \mathbb{Z}}\left(\widetilde{S_{\beta}}\cdot\rho^{i}\widetilde{D_{v}}\right)t^{i}.
\]
\end{claim}

\begin{proof}[Proof of Claim~\ref{claim:pi2-comp}]
In Proposition~\ref{prop:handle-decomp}, we show that (1) there is a handle decomposition of $B^{5}\setminus \mathring N(H_{v})$ consisting of a single $0$-, $1$-, and $2$-handle, where the belt sphere of the 2-handle intersects $Z$ in the disk $D_v$; and (2) there is a relative handle decomposition of $B^{5}\setminus \mathring N(H_{h})$ relative to $Z$ that consists of a single $3$-handle, whose attaching sphere is $S_v\subset Z$.

By Equation~(\ref{eq:glue}), $X$ is obtained by gluing $\overline{B^{5}\setminus \mathring N(H_{h})}$ to $B^{5}\setminus \mathring N(H_{v})$ along the subspace $Z$ of their boundary, via the map $\boldsymbol{\beta_{v}\beta_{h}} : Z\to Z$. Hence, we obtain a handle decomposition of $X$ that consists of a single $0$-, $1$-, $2$-, and $3$-handle, such that in the level set $\partial (B^5 \setminus \mathring N(H_v))$ between the $2$- and $3$-handles, the attaching sphere of the $3$-handle is $\boldsymbol{\beta_{v}\beta_{h}}S_{v}\subset Z$, and the belt sphere of the $2$-handle intersects $Z$ in $D_v$.

Hence, $f$ is the mod $2$ equivariant intersection number of the
attaching sphere of the $3$-handle and the belt sphere of the $2$-handle in $\partial(B^{5}\setminus \mathring N(H_{v}))$, and so Equation~(\ref{eq:pi2tensorf2}) follows.
\end{proof}

\begin{figure}[h]
\begin{centering}
\includegraphics[scale=0.85]{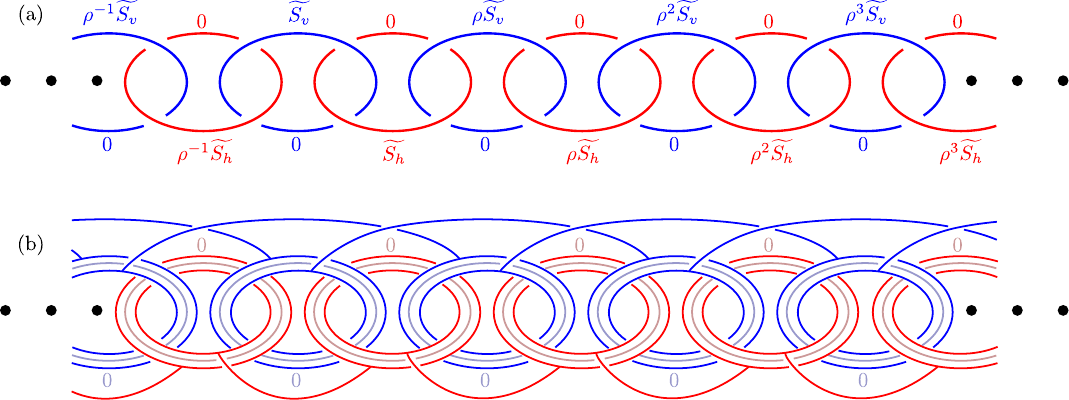}
\par\end{centering}
\caption{\label{fig:ztilde}(a): A handle diagram for the universal cover $\widetilde{Z}$
of $Z=S^{4}\setminus \mathring N(T^{2})$. The $2$-handles correspond to the
$2$-spheres $\rho^{i}\widetilde{S_{h}}$ and $\rho^{i}\widetilde{S_{v}}$.
(b): Lifts of the barbell $\beta_{h,1}$ are drawn in red, and lifts
of the barbell $\beta_{v,2}$ are drawn in blue.}
\end{figure}

Let $\widetilde{D_{v}}$, $\widetilde{S_{v}}$, and $\widetilde{S_{h}}$
be lifts of $D_{v}$, $S_{v}$, and $S_{h}$ to $\widetilde{Z}$ (see
Figure~\ref{fig:ztilde}~(a)), such that
\[
\left|\widetilde{D_{v}}\cap\rho^{i}\widetilde{S_{v}}\right|=\begin{cases}
1 & i=0\\
0 & {\rm otherwise}
\end{cases},\ \left|\widetilde{S_{h}}\cap\rho^{i}\widetilde{S_{v}}\right|=\begin{cases}
1 & i=0,1\\
0 & {\rm otherwise}
\end{cases}.
\]
We will choose a particular lift $\widetilde{S_\beta}$ of $\boldsymbol{\beta_v \beta_h } S_v$ and compute the mod $2$ equivariant intersection number $f$
of $\widetilde{S_{\beta}}$ and $\widetilde{D_{v}}$ by computing
the homology class $[\widetilde{S_{\beta}}]\in H_{2}(\widetilde{Z};\mathbb{F}_{2})$
in terms of $[\rho^{i}\widetilde{S_{h}}]$ and $[\rho^{i}\widetilde{S_{v}}]$.
For simplicity, we drop the brackets from the notation.

Let $\boldsymbol{\widetilde{\beta_{h}}}$ (resp.\ $\boldsymbol{\widetilde{\beta_{v}}}$) be
the lift of the barbell diffeomorphism $\boldsymbol{\beta_{h}}$ (resp.\ $\boldsymbol{\beta_{v}}$) on $Z$
to $\widetilde{Z}$ that fixes the inverse image of $S^{4}\setminus \mathring N(\beta_{h}\sqcup\beta_{v})$,
i.e.\ the complement of a neighborhood of the barbells. Then $\widetilde{\boldsymbol{\beta_{v}}}\widetilde{\boldsymbol{\beta_{h}}}\widetilde{S_{v}}$ is a lift of $\boldsymbol{\beta_v \beta_h } S_v$; let $\widetilde{S_{\beta}}:=\widetilde{\boldsymbol{\beta_{v}}}\widetilde{\boldsymbol{\beta_{h}}}\widetilde{S_{v}}$. 

The barbell $\beta_h$ (resp.\ $\beta_v$) in $Z$ lifts to $\mathbb{Z}$-many barbells in $\widetilde{Z}$; for all $i\in \mathbb{Z}$ we have a barbell covering $\beta_h$ (resp.\ $\beta_v$) whose cuffs are parallel copies of $\rho^{i}\widetilde{S_{h}}$
and $\rho^{k+i}\widetilde{S_{h}}$ (resp.\ $\rho^{i}\widetilde{S_{v}}$
and $\rho^{\ell+i}\widetilde{S_{v}}$).
The diffeomorphism $\widetilde{\boldsymbol{\beta_{h}}}$ (resp.\ $\widetilde{\boldsymbol{\beta_{v}}}$) of $\widetilde{Z}$ is the composition of the barbell diffeomorphisms given by these lifts of $\beta_h$ (resp.\ $\beta_v$). We can describe the homology class of the 2-sphere $\widetilde{\boldsymbol{\beta_{v}}}\widetilde{\boldsymbol{\beta_{h}}}\widetilde{S_{v}}$ explicitly using Equation~(\ref{eq:barbell-action}).

First we describe the homology class of $\widetilde{\boldsymbol{\beta_h}}\widetilde{S_v}$ using Equation~(\ref{eq:barbell-action}). To do so, let us study pairs $(\widetilde{\beta},\widetilde{S})$ such that $\widetilde{\beta}$ is a lift of the barbell $\beta_h$, $\widetilde{S}$ is a cuff of $\widetilde{\beta}$, and $\widetilde{S}$ intersects $\widetilde{S_v}$.
There are four such pairs $(\widetilde{\beta},\widetilde{S})$, and the cuffs of the barbell $\widetilde{\beta}$ of these four pairs are parallel copies of (1) $\rho^{-k-1}\widetilde{S_{h}}$ and
$\rho^{-1} \widetilde{S_{h}}$, (2) $\rho^{-k}\widetilde{S_{h}}$ and
$\widetilde{S_{h}}$, (3) $\rho^{-1} \widetilde{S_{h}}$ and $\rho^{k-1}\widetilde{S_{h}}$,
and (4) $\widetilde{S_{h}}$ and $\rho^{k}\widetilde{S_{h}}$. Hence, for each such pair $(\widetilde{\beta},\widetilde{S})$, $\widetilde{S}$ intersects $\widetilde{S_h}$ transversely at one point, and therefore by Equation~(\ref{eq:barbell-action}),
\begin{equation}
\widetilde{\boldsymbol{\beta_{h}}}\widetilde{S_{v}}=\widetilde{S_{v}}+\rho^{-k-1}\widetilde{S_{h}}+\rho^{-k}\widetilde{S_{h}}+\rho^{k-1}\widetilde{S_{h}}+\rho^{k}\widetilde{S_{h}}.\label{eq:bvsh}
\end{equation}
In fact, by the proof of Lemma \ref{lem:barbelltubing}, the sphere $\widetilde{\boldsymbol{\beta_{h}}}\widetilde{S_{v}}$ is given
by tubing together the five spheres in the right hand side of Equation (\ref{eq:bvsh}).

Now we describe the homology class $\widetilde{\boldsymbol{\beta_{v}}}\widetilde{\boldsymbol{\beta_{h}}}\widetilde{S_{v}}$.
Similarly to the previous paragraph, each of the spheres $\rho^{-k-1}\widetilde{S_{v}},\rho^{-k}\widetilde{S_{h}},\rho^{k-1}\widetilde{S_{h}},\rho^{k}\widetilde{S_{h}}$ in Equation~(\ref{eq:bvsh}) intersects four cuffs of various lifts of the barbell $\beta_{v}$, while $\widetilde{S_v}$ does not intersect any lifts of $\beta_v$. In total, these intersections contribute sixteen new spheres of the form $\rho^{i}\widetilde{S_{v}}$ once we apply $\widetilde{\boldsymbol{\beta_v}}$. Explicitly, we have 
\begin{multline}\label{eq:barbell-sphere-homology-class}
\widetilde{\boldsymbol{\beta_{v}}}\widetilde{\boldsymbol{\beta_{h}}}\widetilde{S_{v}}=\widetilde{S_{v}}+\rho^{-k-1}\widetilde{S_{h}}+\left(\rho^{-k-1-\ell}\widetilde{S_{v}}+\rho^{-k-\ell}\widetilde{S_{v}}+\rho^{-k-1+\ell}\widetilde{S_{v}}+\rho^{-k+\ell}\widetilde{S_{v}}\right)\\
+\rho^{-k}\widetilde{S_{h}}+\left(\rho^{-k-\ell}\widetilde{S_{v}}+\rho^{-k-\ell+1}\widetilde{S_{v}}+\rho^{-k+\ell}\widetilde{S_{v}}+\rho^{-k+\ell+1}\widetilde{S_{v}}\right)\\
+\rho^{k-1}\widetilde{S_{h}}+\left(\rho^{k-1-\ell}\widetilde{S_{v}}+\rho^{k-\ell}\widetilde{S_{v}}+\rho^{k-1+\ell}\widetilde{S_{v}}+\rho^{k+\ell}\widetilde{S_{v}}\right)\\
+\rho^{k}\widetilde{S_{h}}+\left(\rho^{k-\ell}\widetilde{S_{v}}+\rho^{k-\ell+1}\widetilde{S_{v}}+\rho^{k+\ell}\widetilde{S_{v}}+\rho^{k+\ell+1}\widetilde{S_{v}}\right).
\end{multline}
Since we also have $$\rho^i\widetilde{D_v}\cdot \rho^j\widetilde{S_h} = 0 \mathrm{\ and\ }\rho^i\widetilde{D_v}\cdot \rho^j\widetilde{S_v} = \delta_{ij},$$ our final intersection count yields 
\[
f=t^{-k-\ell-1}+t^{-k-\ell+1}+t^{-k+\ell-1}+t^{-k+\ell+1}+1+t^{k-\ell-1}+t^{k-\ell+1}+t^{k+\ell-1}+t^{k+\ell+1},
\]
and so 
\begin{equation}
{\rm rk}_{\mathbb{F}_{2}}\left(\pi_{2}(X)\otimes_{\mathbb{Z}}\mathbb{F}_{2}\right)={\rm rk}_{\mathbb{F}_{2}}H_{2}(\widetilde{X};\mathbb{F}_{2})=2k+2\ell+2\neq0.\label{eq:pi2}\qedhere
\end{equation}
\end{proof}

Interestingly, the following $3$-knots constructed similarly to  $H_{v}\cup\overline{\boldsymbol{\beta_{v,\ell}\beta_{h,k}} H_{h}}$ are unknotted:

\begin{lem}\label{lem:unknots}
The $3$-knots $H_{v}\cup\overline{\boldsymbol{\beta_{v,\ell}} H_{h}},\ H_{v}\cup\overline{\boldsymbol{\beta_{h,k}} H_{h}},\ H_{v}\cup\overline{\boldsymbol{\beta_{h,k}} \boldsymbol{\beta_{v,\ell}} H_{h}} \subset S^5 $
are unknotted.
    
\end{lem}

\begin{proof}
Note that by a result of Levine {\cite[\S 23]{MR266226}} (restated in Theorem~\ref{thm:unknotted-handlebody}), one can show that these 3-knots are unknotted by computing $\pi_{1}$ and $\pi_{2}$
of their exteriors, proceeding as in the proof of Theorem~\ref{thm:morsesimple-s3}. Here, we present a more direct proof that uses a result of Hartman \cite{hartman_unknotting3balls},
which says that any two $3$-balls in $S^{4}$ with the same boundary
become isotopic in $B^{5}$ once their interiors are pushed into $B^{5}$.

For simplicity, we again drop the subscripts $k,\ell$ from the notation. First we show that $\boldsymbol{\beta_{v}}H_{h} \approx H_{h}$ in $B^{5}$ rel.\ $S^{4}$. The barbell $\beta_{v}$ is disjoint from a neighborhood of
some compressing disk $D_{h}$ for $H_{h}$. Let $B$ be the 3-ball $H_{h}\setminus \mathring N(D_{h})\subset S^4$. By \cite{hartman_unknotting3balls}, we can isotope
$\boldsymbol{\beta_{v}}B$ to $B$ rel.\ $S^{4}$ in $B^{5}$, so that $\boldsymbol{\beta_{v}}H_{h}\approx H_{h}$
in $B^{5}$ rel.\ $S^{4}$, and therefore $H_v\cup \overline{\boldsymbol{\beta_v} H_h}\approx H_v\cup\overline{H_h}$ in $S^5$.

Next, to show that $H_{v}\cup\overline{\boldsymbol{\beta_{h}}H_{h}}$ and $H_{v}\cup\overline{\boldsymbol{\beta_{h}\beta_{v}}H_{h}}$
are unknotted, notice that 
\[
H_{v}\cup\overline{\boldsymbol{\beta_{h}}H_{h}}\approx\left(\left(\boldsymbol{\beta_{h}}\right)^{-1}H_{v}\right)\cup\overline{H_{h}},\ H_{v}\cup\overline{\boldsymbol{\beta_{h}\beta_{v}}H_{h}}\approx\left(\left(\boldsymbol{\beta_{h}}\right)^{-1}H_{v}\right)\cup\overline{\boldsymbol{\beta_{v}}H_{h}},
\]
by Equation~(\ref{eq:betainv}). These are unknotted since $\left(\boldsymbol{\beta_{h}}\right)^{-1}H_{v}\approx H_{v}$
in $B^{5}$ rel.\ $S^{4}$ by the same argument as above.
\end{proof}

\begin{thm}[{\cite[\S 23]{MR266226}}]\label{thm:unknotted-handlebody} 
A $3$-knot $Y\subset S^{5}$ is unknotted if and only if $\pi_{1}(S^{5}\setminus Y)\cong\mathbb{Z}$
and $\pi_{2}(S^{5}\setminus Y)=0$.
\end{thm}

% \maketitle
\begin{figure}[h]
\begin{centering}
\includegraphics{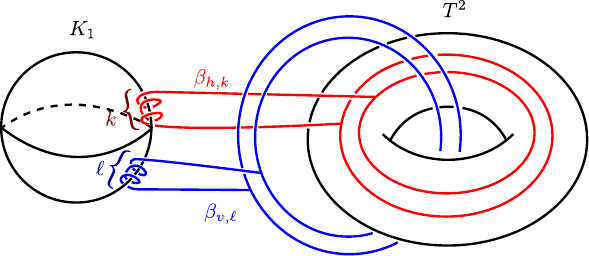}
\par\end{centering}
\caption{\label{fig:tmp3}The trivial link $K_{1}\sqcup T^{2}\subset S^{3}\times0$,
together with the horizontal barbell $\beta_{h,k}$ and the vertical
barbell $\beta_{v,\ell}$ in $S^{4}\setminus(K_{1}\sqcup T^{2})$} 
\end{figure}

Using a similar construction, we now construct, for all $n\ge2$,
infinitely many $n$-component Brunnian 3-links in $S^{5}$ with $(2n+2)$
critical points. 

Let $K_{1}\sqcup\cdots\sqcup K_{n-1}\sqcup T^{2}\subset S^{3}\times0$
be the standard $n$-component split link of $(n-1)$ $2$-spheres
and one $T^{2}$ (see Figure~\ref{fig:tmp3} for $n=2$). For $i=1,\cdots,n-1$,
let $B_{i}$ be the $3$-ball\footnote{According to our earlier convention, this would be the horizontal
handlebody that $K_{i}$ bounds.} that $K_{i}$ bounds in $S^{3}\times\{0\}$ that is disjoint from
$T^{2}$ (and all the other $K_{j}$'s), and let $H_{h}$ (resp.\ $H_{v}$)
be the horizontal (resp. vertical) handlebody that $T^{2}$ bounds
as before, but slightly isotope $H_{v}$ rel.\ $T^{2}$ in $S^{4}$
such that it does not intersect $K_{1}\sqcup\cdots\sqcup K_{n-1}$.
Similarly, let $S_{h}$ (resp.\ $S_{v}$) be the horizontal (resp.\ vertical)
sphere for $T^{2}$. Let $K:=K_{1}\sqcup\cdots\sqcup K_{n-1}$ and
$B:=B_{1}\sqcup\cdots\sqcup B_{n-1}$.

Let $F_{n}$ be the free group on $n$ letters $\rho_{1},\cdots,\rho_{n}$,
and identify $F_{n}\cong\pi_{1}(S^{4}\setminus(K\sqcup T^{2}))$
by identifying the meridian of $K_{i}$ with $\rho_{i}$ and the meridian
of $T^{2}$ with $\rho_{n}$. Our convention for composition in $\pi_{1}$
is that for $\gamma_{1},\gamma_{2}\in\pi_{1}$, their composition
$\gamma_{1}\gamma_{2}\in\pi_{1}$ is represented by the loop given
by concatenating $\gamma_{1}$ and $\gamma_{2}$ in this order (i.e.\ $\gamma_{1}$
comes first). Note that $\pi_{1}$ acts on the universal cover by
deck transformations on the \emph{left}.

Define words $w_{m}\in F_{n}$ for $m\ge1$ recursively by letting
$w_{1}:=\rho_{1}$ and $w_{m+1}:=[w_{m},\rho_{m+1}]$. Let $w:=w_{n-1}\in F_{n}\cong\pi_{1}(S^{4}\setminus(K\sqcup T^{2}))$.
Let the horizontal barbell $\beta_{h,k}$ (resp.\ vertical barbell
$\beta_{v,\ell}$) be such that its cuffs are two parallel copies
of $S_{h}$ (resp.\ $S_{v}$), and its bar loops around $K\sqcup T^{2}$
according to the word $w^{k}$ (resp.\ $w^{\ell}$), and does not
link nontrivially with the cuffs.

\begin{thm}
\label{thm:linked-6crit}Let $k,\ell\ge1$ and let $Y_{k,\ell}:=(B\sqcup H_{v})\cup\overline{\boldsymbol{\beta_{v,\ell}\beta_{h,k}}(B\sqcup H_{h})}\subset B^{5}\cup\overline{B^{5}}=S^{5}$. Then, $Y_{k,\ell}$ is an $n$-component Brunnian $3$-link in $S^{5}$,
i.e.\ $Y_{k,\ell}$ is non-isotopic to the unlink but if we remove
any one of the components, then it becomes isotopic to the unlink
with $(n-1)$ components. Moreover, if $k',\ell'\ge1$ and $\{k,\ell\}\neq\{k',\ell'\}$,
then $Y_{k,\ell}$ and $Y_{k',\ell'}$ are non-isotopic.
\end{thm}

To show that $Y_{k,\ell}$ becomes the trivial $(n-1)$-component unlink if we remove any one of its components, we use the following stronger lemma.

\begin{lem}
\label{lem:s2d2-barbell-trivial}The links of $3$-balls $B$ and
$\boldsymbol{\beta_{v,\ell}\beta_{h,k}}(B)$ are isotopic rel.\ $N(K)$
in $S^{4}$. For all $i\in \{1,\cdots,n-1\}$, the links of handlebodies $(B\setminus B_{i})\sqcup H_{h}$
and $\boldsymbol{\beta_{v,\ell}\beta_{h,k}}((B\setminus B_{i})\sqcup H_{h})$ are isotopic rel.\ $N((K\setminus K_{i})\sqcup T^{2})$ in $S^{4}$. 
\end{lem}

\begin{proof}
First, if we ignore $T^2$ and work in $S^4 \setminus N(K)$, then the cuffs of the barbells $\beta_{h,k}$ and $\beta_{v,\ell}$ bound $3$-balls whose interiors are disjoint from the barbells.
Hence, we can isotope $\beta_{h,k}$ and $\beta_{v,\ell}$ in $S^{4}$ rel.\ $N(K)$
(i.e.\ we allow ourselves to move the $T^{2}$) so that they are
disjoint from $B$. Thus $B$ and $\boldsymbol{\beta_{v,\ell}\beta_{h,k}}B$ are
isotopic rel.\ $N(K)$.

Fix $i\in\{1,\cdots,n-1\}$. For the second statement, we claim that
the diffeomorphisms $\boldsymbol{\beta_{v,\ell}},\boldsymbol{\beta_{h,k}}$
of $S^{4}$ are isotopic rel.\ $N((K\setminus K_{i})\sqcup T^{2})$
to diffeomorphisms of $S^{4}$ supported inside a $B^{4}$ (in fact,
see Remark~\ref{rem:diffeo-isotopic}). Once this is established
the second statement follows, since we can isotope $\boldsymbol{\beta_{h,k}}$
and $\boldsymbol{\beta_{v,\ell}}$ to diffeomorphisms supported
inside a $4$-ball in $S^{4}\setminus \mathring N((K\setminus K_{i})\sqcup T^{2})$,
and then further isotope the diffeomorphisms so that their supports are
contained in a 4-ball disjoint from $(B\setminus B_{i})\sqcup H_{h}$,
keeping $N((K\setminus K_{i})\sqcup T^{2})$ fixed throughout.

Since the word $w$ becomes trivial in $F_{n}/\left\langle \left\langle \rho_{i}\right\rangle \right\rangle $, then
if we ignore $K_{i}$,
the bar of the
barbell $\beta_{h,k}$ (resp.\ $\beta_{v,\ell}$) can be isotoped
rel.\ $N((K\setminus K_{i})\sqcup T^{2})$ until the barbell is
contained inside a closed tubular neighborhood of $S_{h}$ (resp.\ $S_{v}$),
which we identify with an $S^{2}\times D^{2}$ embedded in
$S^{4}\setminus \mathring N(T^{2})$. In the remainder of this paragraph, let
us use $\beta$ to refer to either $\beta_{h,k}$ or $\beta_{v,\ell}$;
the arguments for both are identical. After our initial isotopy of
the bar, we can view $\boldsymbol{\beta}$ as a diffeomorphism of
$S^{2}\times D^{2}$ rel.\ $\partial$.
Let $D:=\mathrm{pt}\times D^{2}$ and $S:=S^{2}\times\mathrm{pt}$.
By Equation~(\ref{eq:barbell-action}),
$[\boldsymbol{\beta}(D),\partial D]=[D,\partial D]+[S]-[S]=[D,\partial D]\in H_{2}(S^{2}\times D^{2},\partial D)$,
so by \cite[Theorem 10.4]{MR4127900}, $\boldsymbol{\beta}(D)\approx D$
rel.\ $\partial D$. After further isotopy, keeping $\partial(S^{2}\times D^{2})$
fixed, we can assume that $\boldsymbol{\beta}$ is the identity on
a neighborhood of $D$ (see, e.g.\ \cite[Proof of Theorem 1.8]{krushkal2024corksexoticdiffeomorphisms}),
so that $\boldsymbol{\beta}$ is isotopic to a diffeomorphism supported
in $S^{2}\times D^{2}\setminus\mathring N((\partial(S^{2}\times D^{2})\cup D))$,
i.e.\ in a $B^{4}\subset S^{2}\times D^{2}$. Since the isotopies
all took place inside $S^{2}\times D^{2}\subset S^{4}\setminus \mathring N(T^{2})$,
the claim follows.
\end{proof}
\begin{rem}
\label{rem:diffeo-isotopic}We note that it is possible to show that
the diffeomorphisms $\boldsymbol{\beta_{v,\ell}}$ and $\boldsymbol{\beta_{h,k}}$
of $S^{4}$ are actually isotopic to the identity rel.\ $N((K\setminus K_{i})\sqcup T^{2})$.
Recall that in the proof, we reduced to studying a specific barbell
diffeomorphism of $S^{2}\times D^{2}$ rel.\ $\partial$. Similarly
to the definition of the barbell map (Section~\ref{sec:Barbells,-knotted-disks,}),
one can view these diffeomorphisms as induced by arc-spinning in $B^{4}$,
and then for instance use the Dax Isomorphism Theorem (see \cite[Theorem 0.3]{gabaiselfrefdisks}); we leave this to the interested reader.
\end{rem}

\begin{proof}[Proof of Theorem~\ref{thm:linked-6crit}]
First by Lemma~\ref{lem:s2d2-barbell-trivial}, $Y_{k,\ell}$ becomes
isotopic to the $(n-1)$-component unlink if we remove any one of
the components. To show that $Y_{k,\ell}$ is a nontrivial link, let
us compute $\pi_{2}(S^{5}\setminus Y_{k,\ell})\otimes_{\mathbb{Z}}\mathbb{F}_{2}$
as a left $\mathbb{F}_{2}[\pi_{1}(S^{5}\setminus Y_{k,\ell})]$-module.

Let $Z:=S^{4}\setminus \mathring N(K\sqcup T^{2})$ and let $\widetilde{Z}$
be its universal cover. Note that $\pi_{1}(S^{5}\setminus Y_{k,\ell})\cong\pi_{1}(Z)\cong F_{n}$.
By abuse of notation, also denote as $\rho_{i}$ for $i=1,\cdots,n-1$ (resp.\ $\rho_{n}$) the
deck transformation of $\widetilde{Z}$ corresponding to the meridian
of $K_{i}$ (resp.\ $T^{2}$).

Similarly to the proof of Theorem~\ref{thm:morsesimple-s3}, define
lifts $\widetilde{D_{v}},\widetilde{S_{h}},\widetilde{S_{v}}\subset\widetilde{Z}$
of $D_{v},S_{h},S_{v}$, respectively, such that for $\rho,\rho'\in F_{n}$, we have
\[
\left|\rho\widetilde{D_{v}}\cap\rho'\widetilde{S_{v}}\right|=\begin{cases}
1 & \rho=\rho'\\
0 & {\rm otherwise}
\end{cases},\ \left|\rho\widetilde{S_{h}}\cap\rho\rho_{n}^{i}\widetilde{S_{v}}\right|=\begin{cases}
1 & i=0,1\\
0 & {\rm otherwise}
\end{cases}.
\]
Also, let $\widetilde{\boldsymbol{\beta_{h}}}$
(resp.\ $\widetilde{\boldsymbol{\beta_{v}}}$) be the lift of the
barbell diffeomorphism $\boldsymbol{\beta_{h}}$ (resp.\ $\boldsymbol{\beta_{v}}$)
on $Z$ to $\widetilde{Z}$ that fixes the inverse image of $S^{4}\setminus \mathring N(\beta_{h}\sqcup\beta_{v})$.

Then, as in the proof of Claim~\ref{claim:pi2-comp}, we can find
a handle decomposition for $S^{5}\setminus Y_{k,\ell}$ that has one
0-handle, $n$ 1-handles, one 2-handle, and one 3-handle (let us ignore
the 4- and 5-handles), such that the attaching sphere of the 3-handle
is $\boldsymbol{\beta_{v}\beta_{h}}S_{v}\subset Z$ and the belt sphere
of the 2-handle intersects $Z$ in $D_{v}$. Hence, we have 
\[
\pi_{2}(S^{5}\setminus Y_{k,\ell})\otimes_{\mathbb{Z}}\mathbb{F}_{2}\cong\mathbb{F}_{2}[F_{n}]\biggl/\left(\mathbb{F}_{2}[F_{n}]\sum_{\rho\in F_{n}}\left(\widetilde{\boldsymbol{\beta_{v}}}\widetilde{\boldsymbol{\beta_{h}}}\widetilde{S_{v}}\cdot\rho\widetilde{D_v}\right)\rho\right)
\]
as left $\mathbb{F}_{2}[F_{n}]\cong\mathbb{F}_{2}[\pi_{1}(S^{5}\setminus Y_{k,\ell})]$-modules.

We compute the $\mathbb{F}_{2}$-homology class of $\widetilde{\boldsymbol{\beta_{v}}}\widetilde{\boldsymbol{\beta_{h}}}\widetilde{S_{v}}$
in $\widetilde{Z}$ similarly to the proof of Theorem~\ref{thm:morsesimple-s3}:
\begin{multline*}
\widetilde{\boldsymbol{\beta_{v}}}\widetilde{\boldsymbol{\beta_{h}}}\widetilde{S_{v}}=\widetilde{S_{v}}+\rho_{n}^{-1}w^{-k}\widetilde{S_{h}}+\left(\rho_{n}^{-1}w^{-k-\ell}\widetilde{S_{v}}+\rho_{n}^{-1}w^{-k}\rho_{n}w^{-\ell}\widetilde{S_{v}}+\rho_{n}^{-1}w^{-k+\ell}\widetilde{S_{v}}+\rho_{n}^{-1}w^{-k}\rho_{n}w^{\ell}\widetilde{S_{v}}\right)\\
+w^{-k}\widetilde{S_{h}}+\left(w^{-k-\ell}\widetilde{S_{v}}+w^{-k}\rho_{n}w^{-\ell}\widetilde{S_{v}}+w^{-k+\ell}\widetilde{S_{v}}+w^{-k}\rho_{n}w^{\ell}\widetilde{S_{v}}\right)\\
+\rho_{n}^{-1}w^{k}\widetilde{S_{h}}+\left(\rho_{n}^{-1}w^{k-\ell}\widetilde{S_{v}}+\rho_{n}^{-1}w^{k}\rho_{n}w^{-\ell}\widetilde{S_{v}}+\rho_{n}^{-1}w^{k+\ell}\widetilde{S_{v}}+\rho_{n}^{-1}w^{k}\rho_{n}w^{\ell}\widetilde{S_{v}}\right)\\
+w^{k}\widetilde{S_{h}}+\left(w^{k-\ell}\widetilde{S_{v}}+w^{k}\rho_{n}w^{-\ell}\widetilde{S_{v}}+w^{k+\ell}\widetilde{S_{v}}+w^{k}\rho_{n}w^{\ell}\widetilde{S_{v}}\right).
\end{multline*}
Hence, if we let $f_{k,\ell}:=1+(\rho_{n}^{-1}+1)(w^{-k}+w^{k})(1+\rho_{n})(w^{-\ell}+w^{\ell})\in\mathbb{F}_{2}[F_{n}]$,
then 
\[
\pi_{2}(S^{5}\setminus Y_{k,\ell})\otimes_{\mathbb{Z}}\mathbb{F}_{2}\cong\mathbb{F}_{2}[F_{n}]/(\mathbb{F}_{2}[F_{n}]f_{k,\ell})=:M_{k,\ell}
\]
as left $\mathbb{F}_{2}[F_{n}]\cong\mathbb{F}_{2}[\pi_{1}(S^{5}\setminus Y_{k,\ell})]$-modules.

Now, we claim that $M_{k,\ell}\neq0$ and that if $\{k,\ell\}\neq\{k',\ell'\}$,
then $M_{k,\ell}\not\cong M_{k',\ell'}$ as left $\mathbb{F}_{2}[F_{n}]$-modules.
For this, we use the following group theoretic lemma.
\begin{lem}
\label{lem:group-theoretic}There exists a group $G$ and a homomorphism $\varphi:F_{n}\to G$
such that if $H\le G$ is the subgroup generated by $\varphi(w)$
and $\varphi(\rho_{n})$, then $H$ is central in $G$ and is isomorphic
to $\mathbb{Z}^{2}$.
\end{lem}

\begin{proof}[Proof of Lemma~\ref{lem:group-theoretic}]
View $F_{n}=F_{n-1}\ast\mathbb{Z}$ where $F_{n-1}$ is the free
group over $\rho_{1},\cdots,\rho_{n-1}$ and $\mathbb{Z}$ has generator
$\rho_{n}$. Recall that in fact $w\in F_{n-1}$. It is sufficient
to find some group $Q$ and some homomorphism $\psi:F_{n-1}\to Q$ such that $\psi(w)$
is central and has infinite order, since we can let $\varphi$ be
the composition $F_{n-1}\ast\mathbb{Z}\to F_{n-1}\times\mathbb{Z}\xrightarrow{\psi\times{\rm Id}}Q\times\mathbb{Z}.$

Such $\psi$ can be obtained by quotienting out $F_{n-1}$ by the
$n$th lower central series of $F_{n-1}$. Alternatively (compare
\cite[page~127]{robinson1996course}), let $U_{n}$ be the group of
upper triangular $(n\times n)$-matrices over $\mathbb{Z}$ whose
diagonal entries are all $1$. Let $E_{i,j}$ be the $(n\times n)$-matrix
whose $(i,j)$th entry is $1$ and the rest are $0$. Then, let $Q:=U_{n}$
and let $\psi:F_{n-1}\to U_{n}$ be such that $\psi (\rho_{i}) = I+E_{i,i+1}$.
Then, one can show that $\psi(w)=I+E_{1,n}$ and so that it is central
and has infinite order.
\end{proof}
Let $\varphi:F_{n}\to G$ be the group homomorphism from Lemma~\ref{lem:group-theoretic}.
Then, it induces a ring homomorphism $\mathbb{F}_{2}[F_{n}]\to\mathbb{F}_{2}[G]$
which we also denote as $\varphi$; view $\mathbb{F}_{2}[G]$ as a
right $\mathbb{F}_{2}[F_{n}]$-module via $\varphi$. We will show
that $\mathbb{F}_{2}[G]\otimes_{\mathbb{F}_{2}[F_{n}]}M_{k,\ell}\neq0$
and that if $\{k,\ell\}\neq\{k',\ell'\}$, then $\mathbb{F}_{2}[G]\otimes_{\mathbb{F}_{2}[F_{n}]}M_{k,\ell}\not\cong\mathbb{F}_{2}[G]\otimes_{\mathbb{F}_{2}[F_{n}]}M_{k',\ell'}$
as left $\mathbb{F}_{2}[G]$-modules. Since $\mathbb{F}_{2}[G]\otimes_{\mathbb{F}_{2}[F_{n}]}M_{k,\ell}\cong\mathbb{F}_{2}[G]/(\mathbb{F}_{2}[G]\varphi(f_{k,\ell}))$,
and since $\varphi(f_{k,\ell})$ is central in $\mathbb{F}_{2}[G]$,
the annihilator of $\mathbb{F}_{2}[G]\otimes_{\mathbb{F}_{2}[F_{n}]}M_{k,\ell}$
is $\mathbb{F}_{2}[G]\varphi(f_{k,\ell})\subset\mathbb{F}_{2}[G]$.

Since $\mathbb{F}_{2}[G]=\bigoplus_{aH\in G/H}a\mathbb{F}_{2}[H]$
and $a\mathbb{F}_{2}[H]\varphi(f_{k,\ell})\subset a\mathbb{F}_{2}[H]$ ($H$ is central in $G$),
we have $\mathbb{F}_{2}[G]\varphi(f_{k,\ell})\cap\mathbb{F}_{2}[H]=\mathbb{F}_{2}[H]\varphi(f_{k,\ell})$.
Hence, it is sufficient to show that $\mathbb{F}_{2}[H]\varphi(f_{k,\ell})\neq\mathbb{F}_{2}[H]$
and that if $\{k,\ell\}\neq\{k',\ell'\}$, then $\mathbb{F}_{2}[H]\varphi(f_{k,\ell})\neq\mathbb{F}_{2}[H]\varphi(f_{k',\ell'})$.
These follow since $\mathbb{F}_{2}[H]$ (which is isomorphic to $\mathbb{F}_{2}[s^{\pm1},t^{\pm1}]$)
is a (commutative) UFD, $\varphi(f_{k,\ell})\in\mathbb{F}_{2}[H]$
is not a unit, and $\varphi(f_{k,\ell})$ is not a product of $\varphi(f_{k',\ell'})$
with a unit. 
\end{proof}

\subsection{\label{subsec:general3mfds} Morse-simple 3-manifolds in $S^5$}

In this subsection, we generalize Theorem~\ref{thm:morsesimple-s3}
to any $3$-manifold with Heegaard genus $1$.
\begin{thm}
\label{thm:genus1-hd}Let $\varphi:S^{4}\to S^{4}$ be an orientation
preserving diffeomorphism that restricts to an orientation preserving
diffeomorphism on $T^{2}$. For sufficiently large $k,\ell>0$, the
two codimension-$2$ submanifolds $Y:=H_{h}\cup\overline{\varphi H_{h}}$
and $Y_{\beta}:=H_{h}\cup\overline{\boldsymbol{\beta_{h,k}\beta_{v,\ell}}\varphi H_{h}}$
are non-isotopic in $S^{5}$. In fact, we can let 
\[
\pi_{2}(S^{5}\setminus Y_{\beta})\otimes_{\mathbb{Z}}\mathbb{F}_{2}
\]
be finite and arbitrarily large.
\end{thm}

\begin{proof}
The computation is the same as the proof of Theorem~\ref{thm:morsesimple-s3}.
Let $X:=S^{5}\setminus \mathring N(Y_{\beta})$, and $Z:=S^{4}\setminus \mathring N(T^{2})$.
Let us drop the subscripts $k,\ell$ from the notation and define $\widetilde{Z},\rho,\widetilde{D_{h}},\widetilde{S_{h}},\widetilde{S_{v}},\widetilde{\beta_{h}},\widetilde{\beta_{v}}$
as in the proof of Theorem~\ref{thm:morsesimple-s3}.

Similarly to Claim~\ref{claim:pi2-comp}, 
\[
\pi_{2}(S^{5}\setminus Y_{\beta})\otimes_{\mathbb{Z}}\mathbb{F}_{2}\cong\mathbb{F}_{2}[t,t^{-1}]/(f)
\]
where $f$ is the mod $2$ equivariant intersection number of $\boldsymbol{\beta_{h,k}\beta_{v,\ell}}\varphi S_{v}$
and $D_{h}$ in $Z$.

Let $\widetilde{\varphi S_{v}}$ be a lift of $\varphi S_{v}$ to
$\widetilde{Z}$, and let $h_{i}=\widetilde{\varphi S_{v}}\cdot\rho^{i}\widetilde{S_{h}}$,
$v_{i}=\widetilde{\varphi S_{v}}\cdot\rho^{i}\widetilde{S_{v}}$,
and $b_{i}=\widetilde{\varphi S_{v}}\cdot\rho^{i}\widetilde{D_{h}}$.
Then, as elements
of $H_{2}(\widetilde{Z};\mathbb{F}_{2})$, we have
\[
\widetilde{\boldsymbol{\beta_{v}}}\widetilde{\varphi S_{v}}=\widetilde{\varphi S_{v}}+\sum_{i}v_{i}\left(\rho^{-\ell+i}\widetilde{S_{v}}+\rho^{\ell+i}\widetilde{S_{v}}\right).
\]
Hence,
\begin{multline*}
\widetilde{\boldsymbol{\beta_{h}}}\widetilde{\boldsymbol{\beta_{v}}}\widetilde{\varphi S_{v}}=\widetilde{\varphi S_{v}}+\sum_{i}h_{i}\left(\rho^{-k+i}\widetilde{S_{h}}+\rho^{k+i}\widetilde{S_{h}}\right)\\
+\sum_{i}v_{i}\biggl(\rho^{-\ell+i}\widetilde{S_{v}}+\rho^{-k-1-\ell+i}\widetilde{S_{h}}+\rho^{-k-\ell+i}\widetilde{S_{h}}+\rho^{k-1-\ell+i}\widetilde{S_{h}}+\rho^{k-\ell+i}\widetilde{S_{h}}\\
+\rho^{\ell+i}\widetilde{S_{v}}+\rho^{-k-1+\ell+i}\widetilde{S_{h}}+\rho^{-k+\ell+i}\widetilde{S_{h}}+\rho^{k-1+\ell+i}\widetilde{S_{h}}+\rho^{k+\ell+i}\widetilde{S_{h}}\biggl).
\end{multline*}
Therefore, 
\begin{multline*}
H_{2}(\widetilde{X};\mathbb{F}_{2})\cong\mathbb{F}_{2}[t,t^{-1}]\biggl/\sum_{i}\left(\widetilde{\boldsymbol{\beta_{h}}}\widetilde{\boldsymbol{\beta_{v}}}\widetilde{\varphi S_{v}}\cdot\rho^{i}\widetilde{D_{h}}\right)t^{i}=\mathbb{F}_{2}[t,t^{-1}]\biggl/\biggl(\sum_{i}b_{i}t^{i}+h_{i}(t^{-k+i}+t^{k+i})\\
+v_{i}(t^{-k-1-\ell+i}+t^{-k-\ell+i}+t^{k-1-\ell+i}+t^{k-\ell+i}+t^{-k-1+\ell+i}+t^{-k+\ell+i}+t^{k-1+\ell+i}+t^{k+\ell+i})\biggl).
\end{multline*}

Let $M_b :=\max \{ 0, \max \{ |i|: b_i \neq 0 \}\}$, $M_h :=\max \{ 0, \max \{ |i|: h_i \neq 0 \}\}$, and  $M_v :=\max \{ 0, \max \{ |i|: v_i \neq 0 \}\}$. Since $\widetilde{\varphi S_{v}}$ is compact, these are well-defined. The second homology
$H_{2}(\widetilde{Z};\mathbb{F}_{2})$ is freely generated by $\rho^{i}\widetilde{S_{h}}$
and $\rho^{i}\widetilde{S_{v}}$ and the homology class of $\widetilde{\varphi S_{v}}$
is nonzero, and hence there exists some $i$ such that either $h_{i}\neq0$
or $v_{i}\neq0$. If $k\ge M_b + M_h +100$ and $\ell\ge M_b + M_h + M_v +100$,
then 
\[
\dim_{\mathbb{F}_{2}}H_{2}(\widetilde{X};\mathbb{F}_{2})=\begin{cases}
2k+\max\{i:h_{i}\neq0\}-\min\{i:h_{i}\neq0\} & {\rm if}\ v_{i}=0\ \forall i\\
2k+2\ell+1+\max\{i:v_{i}\neq0\}-\min\{i:v_{i}\neq0\} & {\rm otherwise}
\end{cases}.\qedhere
\]
\end{proof}

\begin{rem}
The case where there exists some $i$ such that $h_{i}\neq0$ is simpler:
we only need one barbell. Let $Y_{\beta}':=H_{h}\cup\overline{\boldsymbol{\beta_{h,k}} \varphi H_{h}}$.
For $k\ge M_b + M_h +100$, we have 
\[
\dim_{\mathbb{F}_{2}}\pi_{2}(S^{5}\setminus Y_{\beta}')\otimes_{\mathbb{Z}}\mathbb{F}_{2}=2k+\max\{i:h_{i}\neq0\}-\min\{i:h_{i}\neq0\}.
\]
In particular, if we let $\varphi = \rm Id$, i.e.\ $Y=H_{h}\cup\overline{H_{h}}$ and
$Y_{\beta}'=H_{h}\cup\overline{\boldsymbol{\beta_{h,k}} H_{h}}$ in $S^{5}$, then $Y_\beta '$ is a knotted $S^1 \times S^2$ in $S^5$ with four critical points, since we have
\[
\pi_{2}(S^{5}\setminus Y)\cong\mathbb{Z}[t,t^{-1}],\ {\rm rk}_{\mathbb{F}_{2}}\left(\pi_{2}(S^{5}\setminus Y_{\beta}')\otimes_{\mathbb{Z}}\mathbb{F}_{2}\right)=2k+1.
\]
Note that this also gives an alternative proof of the second statement of Corollary~\ref{cor:bgconj}.
\end{rem}

\begin{cor}
\label{cor:morsesimple3mfd}If $Y$ is $S^{3}$, a lens space, or
$S^{1}\times S^{2}$, then there exist infinitely many pairwise non-isotopic
embeddings of $Y$ in $S^{5}$ that each have exactly four critical
points with respect to the standard
height function of $S^{5}$ (which is Morse on $Y$).
\end{cor}

\begin{proof}
We use Montesinos's theorem \cite[Theorem 5.4 and Proposition 4.3]{MR698205}
to show that for any such choice of $Y$, there exists a diffeomorphism
$\varphi:(S^{4},T^{2})\to(S^{4},T^{2})$ such that $H_{h}\cup\overline{\varphi H_{h}}$ and $H_{h}\cup\overline{\boldsymbol{\beta_{h}\beta_{v}}\varphi H_{h}}$
are diffeomorphic to $Y$.

Let $\mu$ (resp.\ $\lambda$) be the simple closed curve on $T^{2}$
(well-defined up to isotopy and reversing orientation) that bounds
a compressing disk in the horizontal handlebody (resp.\ the vertical
handlebody). Let $\psi:T^{2}\to T^{2}$ be a diffeomorphism, and let
the induced map $\psi_{\ast}:H_{1}(T^{2})\to H_{1}(T^{2})$ be $\begin{pmatrix}a & b\\
c & d
\end{pmatrix}$ with respect to the basis $[\mu],[\lambda]$. Montesinos shows
that $\psi$ extends to a diffeomorphism $\varphi$ of $S^{4}$ if and only if
$a+b+c+d$ is even.

Now, we are left to find, for each choice of $Y$, $a,b,c,d\in \mathbb{Z}$ such that $a+b+c+d$ is even, $ad-bc=1$, and the induced $\varphi$ is such that $H_{h}\cup\overline{\varphi H_{h}}$ and $H_{h}\cup\overline{\boldsymbol{\beta_{h}\beta_{v}}\varphi H_{h}}$
are diffeomorphic to $Y$.
For $Y=S^{3}$, take $\psi_{\ast}=\begin{pmatrix}0 & -1\\
1 & 0
\end{pmatrix}$, and for $Y=S^{1}\times S^{2}$,
let $\psi_{\ast}={\rm Id}$.

The remaining cases are the lens spaces: for coprime $p,q\in\mathbb{Z}\setminus\{0\}$,
let $Y$ be the lens space represented by a Heegaard diagram $(T^{2},\alpha,\beta)$,
where $[\alpha]=[\mu]$ and $[\beta]=q[\mu]+p[\lambda]$. First, we
may assume that $p+q$ is odd; if not, then $p$ and $q$ are both
odd, so we can work with $(p,p+q)$ instead of $(p,q)$, since $L(p,q)\cong L(p,p+q)$.

Now since $p$ and $q$ are coprime, there exist $p^{\ast},q^{\ast}\in\mathbb{Z}$
such that $pp^{\ast}-qq^{\ast}=1$. Since $p+q$ is odd, either $\psi_{\ast}=\begin{pmatrix}p & q^{\ast}\\
q & p^{\ast}
\end{pmatrix}$ or $\psi_{\ast}=\begin{pmatrix}p & q^{\ast}+p\\
q & p^{\ast}+q
\end{pmatrix}$ will work.
\end{proof}

\begin{figure}[h]
\begin{centering}
\includegraphics{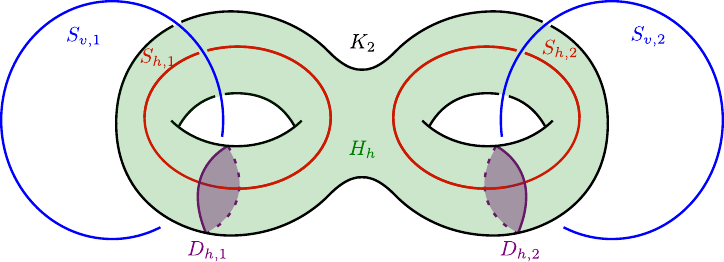}
\par\end{centering}
\caption{\label{fig:handlebody-genus2-disks}The standard genus $2$ surface $K_2 \subset S^3 \times 0 \subset S^4 $, the horizontal handlebody $H_h \subset S^3 \times 0 \subset S^4$ that $K_2$ bounds, $2$-spheres $S_{h,s},S_{v,s} \subset S^4 \setminus \mathring {N}(K_2)$ for $s=1,2$
whose homology classes generate $H_{2}(S^{4}\setminus K_{g}) \cong\mathbb{Z}^{4}$, and compressing disks  $D_{h,s}$ be two compressing disks of $H_h$ such that $D_{h,s}\cap S_{h,r}=\delta_{sr}$.}
\end{figure}

Lastly, we make the following computation for the purposes of proving Corollary \ref{cor:simple-5d} in Subsection \ref{subsec:handlebodycorollaries}. Let $K_2 , H_h , S_{h,s}, S_{v,s}, D_{h,s}$ for $s=1,2$ be as in Figure~\ref{fig:handlebody-genus2-disks}.

\begin{thm}
\label{thm:simple-5d}Let $\beta$
be a barbell in $S^{4}\setminus K_{2}$ whose two cuffs are $S_{h,1}$ and $S_{h,2}$, and let $Y:=H_{h}\cup\overline{H_{h}}$ and $Y_{k}:=H_{h}\cup\overline{\boldsymbol{\beta}^{k}H_{h}}$.
Then, as $\mathbb Z [t,t^{-1}]\cong \mathbb Z [\pi_1 (S^5 \setminus Y)] \cong \mathbb Z [\pi _1 (S^5 \setminus Y_k )]$-modules, 
\[
\pi_{2}(S^5 \setminus Y)\cong\mathbb{Z}[t,t^{-1}]\oplus\mathbb{Z}[t,t^{-1}],\ \pi_{2}(S^5 \setminus Y_{k})\cong\mathbb{Z}[t,t^{-1}]/(k(t-1))\oplus\mathbb{Z}[t,t^{-1}]/(k(t-1)),
\]
and so $Y,Y_{1},Y_{2},\cdots$ are pairwise non-isotopic embeddings of $S^1 \times S^2 $ in $S^5$.
\end{thm}

\begin{proof}
As in the proof of Theorem~\ref{thm:morsesimple-s3}, let $X_{k}=S^{5}\setminus \mathring N(Y_{k})$,
$Z=S^{4}\setminus \mathring N(K_{2})$, let $\widetilde{X_{k}}$ and $\widetilde{Z}$
be the universal covers of $X_{k}$ and $Z$, respectively, let $\rho:\widetilde{Z}\to\widetilde{Z}$
be the deck transformation corresponding to the meridian of $K_{2}$, and let $\widetilde{\boldsymbol{\beta}}:\widetilde{Z}\to\widetilde{Z}$ be
the lift of the barbell diffeomorphism $\boldsymbol{\beta}:Z\to Z$ that fixes the inverse
image of $S^{4}\setminus \mathring N(\beta)$.

For $s=1,2$, let $\widetilde{D_{h,s}}$, $\widetilde{S_{h,s}}$,
and $\widetilde{S_{v,s}}$ be lifts of $D_{h,s}$, $S_{h,s}$,
and $S_{v,s}$ to $\widetilde{Z}$, such that
\[
\left|\widetilde{D_{h,s}}\cap\rho^{i}\widetilde{S_{h,s}}\right|=\begin{cases}
1 & i=0\\
0 & {\rm otherwise}
\end{cases},\ \left|\widetilde{S_{h,s}}\cap\rho^{i}\widetilde{S_{v,s}}\right|=\begin{cases}
1 & i=0,1\\
0 & {\rm otherwise}
\end{cases},
\]
and such that there is a lift of the barbell $\beta$ whose cuffs are $\widetilde{S_{h,1}}$
and $\widetilde{S_{h,2}}$ (Figure \ref{fig:genus2-handlebody-5d}).

Similarly to Claim~\ref{claim:pi2-comp}, we can show that 
\[
\pi_{2}(X_{k})\cong H_{2}(\widetilde{X_{k}})\cong{\rm coker}\left(F:=\begin{pmatrix}f_{11} & f_{12}\\
f_{21} & f_{22}
\end{pmatrix}:\mathbb{Z}[t,t^{-1}]\oplus\mathbb{Z}[t,t^{-1}]\to\mathbb{Z}[t,t^{-1}]\oplus\mathbb{Z}[t,t^{-1}]\right)
\]
where 
\[
f_{rs}:=\sum_{i}\left({\widetilde{\boldsymbol{\beta}}}^{k}\widetilde{S_{v,s}}\cdot\rho^{i}\widetilde{D_{h,r}}\right)t^{i}.
\]
Indeed, we can find a handle decomposition of $X_k$ that has a 1-handle, two 2-handles, and two 3-handles, where the belt spheres of the 2-handles intersect $Z$ in $D_{h,1}$ and $D_{h,2}$, and the attaching spheres of the 3-handles are $\boldsymbol{\beta}^kS_{v,1}$ and $\boldsymbol{\beta}^kS_{v,2}$.

We can
compute the homology classes of ${\widetilde{\boldsymbol{\beta}}}^{k}\widetilde{S_{v,1}}$
and ${\widetilde{\boldsymbol{\beta}}}^{k}\widetilde{S_{v,2}}$:
\[
{\widetilde{\boldsymbol{\beta}}}^{k}\widetilde{S_{v,1}}=\widetilde{S_{v,1}}-k\rho^{-1}\widetilde{S_{h,2}}+k\widetilde{S_{h,2}},\ {\widetilde{\boldsymbol{\beta}}}^{k}\widetilde{S_{v,2}}=\widetilde{S_{v,2}}+k\rho^{-1}\widetilde{S_{h,1}}-k\widetilde{S_{h,1}}\in H_{2}(\widetilde{X_{k}}).
\]
Hence, 
\[
F=\begin{pmatrix}0 & k-kt^{-1}\\
kt^{-1}-k & 0
\end{pmatrix}.\qedhere
\]
\end{proof}

\begin{figure}
    \centering
    \includegraphics[scale=0.7]{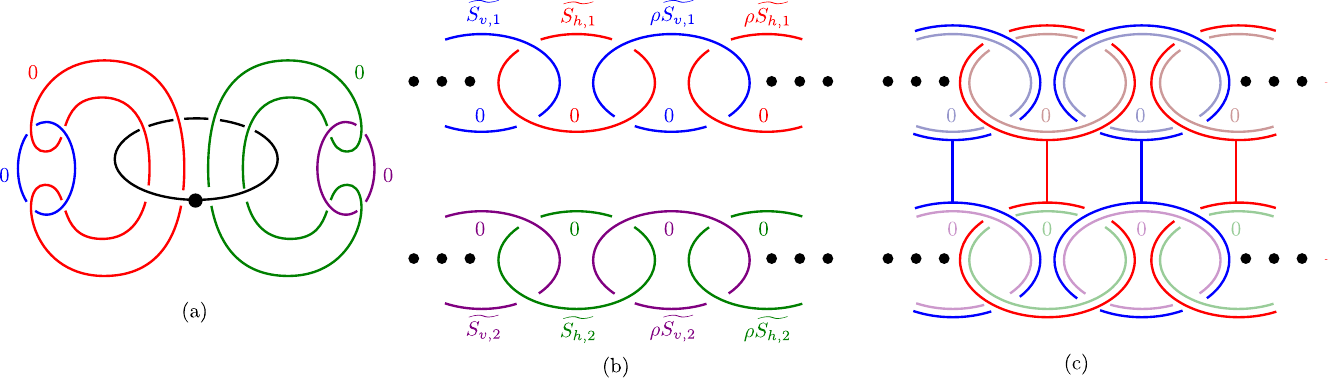}
    \caption{(a): A handle diagram for $Z=S^4 \setminus \mathring{N}(K_2)$. (b): The universal cover $\widetilde{Z}$. (c): The lifted barbells in $\widetilde{Z}$.}
    \label{fig:genus2-handlebody-5d}
\end{figure}

\subsection{\label{subsec:handlebodycorollaries} Knotted and linked handlebodies in $B^5$}

In this subsection we collect various corollaries of Subsections \ref{subsec:3knots3links} and \ref{subsec:general3mfds} concerning knotted and linked handlebodies in $S^4$ that remain knotted and linked in $B^5$. In general, the methods used to construct the knotted and linked 3-manifolds in Subsections \ref{subsec:3knots3links} and \ref{subsec:general3mfds} naturally give rise to such handlebodies, which we note in Lemma \ref{lem:stabilization}. Moreover, the knotted and linked handlebodies that arise in this way also remain knotted and linked after any number of \textit{stabilizations} in $S^4$.

\begin{defn}[Stabilization]\label{def:stabilization}
Let $B:= D^2 \times  [-1,1]\times [-1,1]$, $H_{loc}:= D^2 \times [-1,0] \times 0 \subset B$, and let $H_{loc}^\mathrm{st}$ be a boundary connected sum of  $H_{loc}$ with a genus $1$ handlebody inside $B$. This is our local model; note that any two choices of $H_{loc}^\mathrm{st}$ are isotopic in $B$ rel.\ $\partial B$.

Let $H_1, H_2, \cdots $ be a collection of (possibly disconnected) handlebodies in $S^4$ with common boundary $K$. Assume that there exists a small $4$-ball $B=D^2 \times [-1,1]\times [-1,1]\subset S^4$ such that  $K\cap B = D^2 \times 0\times 0$ and $H_i \cap B=H_{loc}$ for all $i$. (Note that we can always achieve this after isotoping the $H_i$'s rel.\ $K$ near a point on $K$.) The \emph{stabilizations} of $H_1, H_2 , \cdots $ are obtained by replacing $H_{loc} \subset B$ by $H_{loc}^\mathrm{st}$.
\end{defn}

\begin{lem}\label{lem:stabilization} Let $H_1, H_2, H_3$ be three (possibly disconnected) handlebodies in $S^4$ with common boundary $K\subset S^4$. Let $H_1'$ and $H_2'$ be handlebodies obtained from $H_1 $ and $H_2$ by any number of stabilizations. Push the interiors of $H_i$ into $B^5$ and denote these also by $H_i$. Suppose that the 3-manifolds $Y_1 := H_{3}\cup\overline{H_{1}}$ and $Y_2 := H_{3}\cup\overline{H_{2}}\subset S^{5}$
are non-isotopic in $S^5$. Then $H_1'\not\approx H_2'$~rel.~$\partial$, even when their interiors are pushed into $B^5$.\end{lem}

\begin{proof}
The fact that $H_1\not\approx H_2$ in $B^5$ rel.\ $K$ is immediate, since if they were isotopic in $B^5$ rel.\ $K$, $Y_1$ and $Y_2$ would be isotopic in $S^5$. 

We first show that $H_1$ and $H_2$ survive one stabilization.
Let $B=D\times [-1,1]\times [-1,1]\subset S^4$ be as in Definition~\ref{def:stabilization}. Isotope $H_{3}$ rel.\ $K$ near $B$ such
that we also have $H_{3}\cap B=D\times[-1,0]\times 0$. Let $H_{h}\subset D\times[0.1,1]\times[0,1]$
and $H_{v}\subset D\times[0.1,1]\times[-1,0]$ be genus $1$ handlebodies
with the same boundary such that 
\[
H_{h}\cap D\times[0.1,1]\times0=\partial H_h = \partial H_v = H_{v}\cap D\times[0.1,1]\times0
\]
and $H_{v}\cup\overline{H_{h}}$ is a standard $3$-sphere in $D\times[0.1,1]\times[-1,1]$.

Consider the stabilizations $H_{i} ^{\mathrm{st}} :=H_{i}\natural H_{h}$ for $i=1,2$
and $H_{3} ^{\mathrm{st}} :=H_{3}\natural H_{v}$, where we take the boundary
sum along the same arc, such that that the boundaries $\partial H_i ^{\mathrm{st}}$ coincide for $i=1,2,3$. 

Now, push the interiors of $H_{i}^{\mathrm{st}}$ into $B^{5}$ and form the
3-manifolds $Y_{1}^\mathrm{st}:=H_{3}^\mathrm{st}\cup\overline{H_{1}^\mathrm{st}}$ and $Y_{2}^\mathrm{st}:=H_{3}^\mathrm{st}\cup\overline{H_{2}^\mathrm{st}}$
in $S^{5}=B^{5}\cup\overline{B^{5}}$. Then $Y_{i}^\mathrm{st}=(H_{3}\cup\overline{H_{i}})\#(H_{v}\cup\overline{H_{h}}) = Y_i \#(H_{v}\cup\overline{H_{h}}) $ for $i=1,2$. Since $H_{v}\cup\overline{H_{h}}$ is an unknotted $S^{3}$ in $S^{5}$,
$Y_{i}^\mathrm{st}$ and $Y_{i}$ are isotopic.

This shows that the push-ins $H_1 ^\mathrm{st}$ and $H_2 ^\mathrm{st}$ are non-isotopic rel.\ $\partial$. Since $H_i ^\mathrm{st}$ for $i=1,2,3$ satisfy the conditions of the lemma, we conclude that $H_1$ and $H_2$ remain non-isotopic after any number of stabilizations.\end{proof}

Lemma \ref{lem:stabilization} implies that all the handlebodies we used to construct the 3-manifolds of Subsections \ref{subsec:3knots3links} and \ref{subsec:general3mfds} are non-isotopic in $B^5$ and after any number of stabilizations. In particular, the genus $1$ handlebodies used to construct the 3-knots of Theorem \ref{thm:morsesimple-s3} resolve \cite[Conjecture 11.3]{budney2021knotted3ballss4}. The knotted handlebodies of Theorem \ref{thm:simple-knotted-handlebody} in Section \ref{sec:Simple-barbell-implantations} and Theorem \ref{thm:genus1-handlebody} in Section \ref{sec:More-complicated-barbells} also remain non-isotopic when their interiors are pushed into $B^5$.  For completeness, we collect these knotted handlebodies below and reproduce their figures in Figure \ref{fig:section53-recap}.

\begin{figure} 
    \centering
    \includegraphics[scale=0.49]{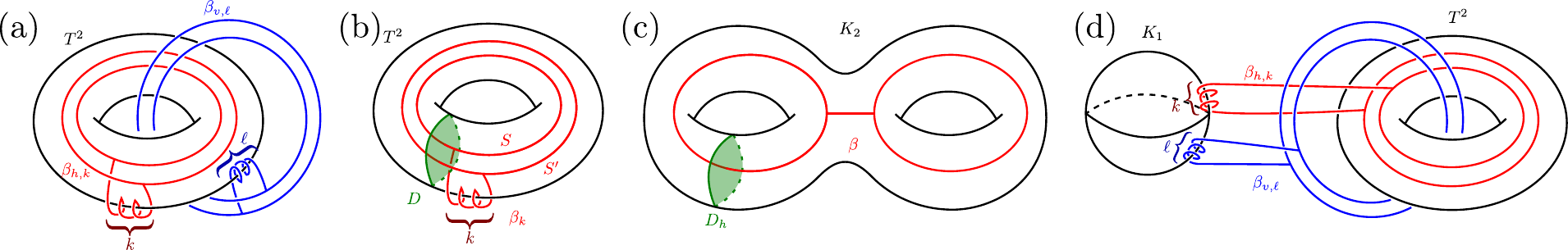}
    \caption{(a): The barbells $\beta_{h,k}, \beta_{v,{\ell}}$ used to construct the 3-knots of Theorem~\ref{thm:morsesimple-s3}. (b): The barbell $\beta_k$ used to construct the knotted solid tori of Theorem~\ref{thm:genus1-handlebody}. (c): The barbell $\beta$ used to construct the knotted genus $2$ handlebodies of Theorem~\ref{thm:simple-knotted-handlebody} and the knotted embeddings of $S^1 \times S^2$ of Theorem~\ref{thm:simple-5d}. (d): The barbells $\beta_{h,k}, \beta_{v,{\ell}}$ used to construct the Brunnian 3-links of Theorem~\ref{thm:linked-6crit}.}
    \label{fig:section53-recap}
\end{figure}

\begin{cor}
\label{cor:bgconj} The genus $1$ handlebodies $\{H_{h}\}\cup\{\boldsymbol{\beta_{v,1}\beta_{h,k}}H_{h}\}_{k\ge1}$
in $S^{4}$  are pairwise non-isotopic rel.\ $\partial$. Moreover, they
stay non-isotopic rel.\ $\partial$ even when their interiors are pushed into $B^5$. (Figure~\ref{fig:section53-recap}~(a))

The genus $1$ handlebodies $\{H_{h}\}\cup\{\boldsymbol{\beta_{h,k}}H_{h}\}_{k\ge1}$ from Theorem~\ref{thm:genus1-handlebody} stay non-isotopic rel.\ $\partial$ even when their interiors are pushed into $B^{5}$. (Figure~\ref{fig:section53-recap}~(b))

The genus $g\ge 1$ handlebodies
obtained by stabilizing these genus $1$ handlebodies $(g-1)$ times are also pairwise non-isotopic rel.\ $\partial$, and remain non-isotopic rel.\ $\partial$ even when their interiors are pushed into $B^5$.
\end{cor}

\begin{proof}
The first statement follows from applying Lemma \ref{lem:stabilization} to the 3-knots $H_v\cup\overline{ \boldsymbol{\beta_{v,1}\beta_{h,k}}H_h}$ constructed in Theorem~\ref{thm:morsesimple-s3}. The second statement also follows from Theorem~\ref{thm:morsesimple-s3}, since we have
$$H_v \cup \overline{\boldsymbol{\beta_{v,1}\beta_{h,k}}H_{h}} \approx \left( \boldsymbol{\beta_{v,1}} \right) ^{-1}H_v \cup \overline{\boldsymbol{\beta_{h,k}}H_{h}}$$ by Equation~(\ref{eq:betainv}) and $$H_v \cup \overline{\boldsymbol{\beta_{v,1}} H_{h}} \approx \left( \boldsymbol{\beta_{v,1}} \right) ^{-1}H_v \cup \overline{H_{h}}$$ is isotopic to the unknot by Lemma~\ref{lem:unknots}. The third statement follows from Lemma~\ref{lem:stabilization}.
\end{proof}

\begin{rem} 
\label{rem:genus-1-handlebody-from-3knot} It actually follows from results of \cite{carrara2001} and \cite{hughes2024knotted} that the existence of infinitely many knotted 3-knots in $S^5$ with 4 critical points automatically implies the existence of infinitely many knotted solid tori in $S^4$ that remain knotted after their interiors are pushed into $B^5$, even if those 3-knots were not constructed as an explicit union of two solid tori (as were the 3-knots of Theorem \ref{thm:morsesimple-s3}). To see this, assume that there exists a 3-knot $Y$ in $S^5$ with four critical points with respect to the standard height function. By \cite[Remark 5.2]{carrara2001}
(or by Theorem~\ref{thm:unknotted-handlebody}), those critical points must be distributed as one critical point of each index $0,1,2,3$. After isotopy, we may assume that $Y$ is embedded in $S^5=B^5 \cup \overline{B^5}$ as $Y=H_S \cup \overline{H_N}$, where $H_S$ is the solid torus in $B^5$ given by the $0$ and $1$-handles of $Y$, and $H_N$ is the solid torus in $B^5$ given by the $2$ and $3$-handles of $Y$. By \cite[Lemma 4.1]{hughes2024knotted}, these solid tori in $B^5$ are boundary parallel, i.e.\ they can be isotoped rel.\ $\partial$ into $S^4$. 

Now, we argue as in the fourth and fifth paragraphs of the proof of Theorem 1.3 in \cite{hughes2024knotted}: it follows from Montesinos's theorem \cite[Theorem~5.4 and Proposition~4.3]{MR698205} (compare \cite{MR1928177} and \cite[Theorem~3.5]{hughes2024knotted}) that after a diffeomorphism of $S^4$, we can assume $H_S$ (resp.\ $H_N$) is compressing-curve equivalent to the vertical handlebody $H_v$ (resp.\ the horizontal handlebody $H_h$). However, since $Y\not \approx H_v\cup \overline{H_h}$, it follows that either $H_S\not \approx H_v$ rel.\ $\partial$ or $H_N \not \approx H_h$ rel.\ $\partial$, even after pushing their interiors into $B^5$. In either case, since $H_v$ and $H_h$ are also related by a diffeomorphism of $S^4$, we can obtain a solid torus in $S^4$ that is not isotopic to $H_h$ rel.\ $\partial$ even in $B^5$. The same argument shows that the existence of infinitely many knotted solid tori in $S^4$ already follows from Theorem \ref{thm:morsesimple-s3}, i.e.\ the existence of infinitely many 3-knots with four critical points.
\end{rem}

\begin{cor}\label{cor:simple-5d}
The genus 2 handlebodies $H_h  \cup \{ \boldsymbol{\beta} ^k H_{h}\}_{k\ge 1}$ from Theorem~\ref{thm:simple-knotted-handlebody} 
stay non-isotopic rel.\ $\partial$ even when their interiors are pushed into $B^5$. (Figure~\ref{fig:section53-recap}~(c))

The genus $g\ge 2$ handlebodies obtained by stabilizing these genus $2$ handlebodies are also pairwise non-isotopic rel.\ $\partial$, and remain non-isotopic rel.\ $\partial$ even when their interiors are pushed into $B^5$.
\end{cor}

\begin{proof}
Apply Lemma~\ref{lem:stabilization} to the embeddings $H_h\cup \overline{\boldsymbol{\beta}^kH_h}$ of $S^1 \times S^2$ in $S^5$ from Theorem~\ref{thm:simple-5d}.
\end{proof}

\begin{cor}\label{cor:solid-torus-ball-brunnian} 
      The $n$-component genus $(0,\cdots,0,1)$ handlebody links $\{ B\sqcup H_h\} \cup \{\boldsymbol{\beta_{v,1} \beta_{h,k}}(B\sqcup H_h)\}_{k\ge 1}$ of $(n-1)$ 3-balls and a solid torus in $S^4$ from Theorem \ref{thm:linked-6crit} are pairwise non-isotopic rel.~$\partial$ and remain non-isotopic rel.\ $\partial$ even when their interiors are pushed into $B^5$. (Figure~\ref{fig:section53-recap}~(d) for the $n=2$ case)
    
    For $g_1,\cdots,g_{n-1}\ge 0$ and $g_n\ge 1$, the $n$-component genus $(g_1,\cdots,g_n)$ handlebody links of genus $g_i$ handlebodies obtained by stabilizing these links are also non-isotopic rel.\ $\partial$ and remain so even when their interiors are pushed into $B^5$. 
    
    Finally, for any $i=1,\cdots ,n$, if we remove the $i$th components from these handlebody links, then they become isotopic rel.\ $\partial$ (hence they are Brunnian).
\end{cor}

\begin{proof}
    The first statement follows from applying Lemma \ref{lem:stabilization} to the $n$-component 3-links $(B\sqcup H_h)\cup \overline{\boldsymbol{\beta_{v,{\ell}}\beta_{h,k}}(B\sqcup H_h)}$ of Theorem \ref{thm:linked-6crit}. The second statement follows from Lemma \ref{lem:stabilization}.
    
    Finally, the last statement follows from Lemma \ref{lem:s2d2-barbell-trivial}, where we showed that $B \approx \boldsymbol{\beta_{v,1} \beta_{h,k}}B$ rel.\ $N(\partial B)$ and for each $i=1,\cdots,n-1$, $(B\setminus B_{i})\sqcup H_{h} \approx \boldsymbol{\beta_{v,\ell}\beta_{h,k}}((B\setminus B_{i})\sqcup H_{h})$ rel.\ $N(\partial ((B\setminus B_{i})\sqcup H_{h}))$.
    To see this, observe that starting from the genus $(0,\cdots,0,1)$ Brunnian links $\boldsymbol{\beta_{v,1} \beta_{h,k}}(B\sqcup H_h)$, we can stabilize any of the 3-balls $\boldsymbol{\beta_{v,1} \beta_{h,k}}B_i$ (resp.\ $\boldsymbol{\beta_{v,1} \beta_{h,k}}H_h$) inside 4-balls $D\subset N(\partial B_i)$ (resp.\ $D \subset N(\partial H_h)$).
    Then the above isotopies fix the 4-ball $D$ pointwise, and so the stabilized links of handlebodies are isotopic rel.\ neighborhoods of their boundaries as well.
\end{proof}

Finally, we remark that it is possible to give a $5$-dimensional proof of the existence of infinitely many knotted splitting spheres of the 2-component unlink of unknotted surfaces of genus $m,n$ for $m\ge 0, n\ge 1$; compare \cite{hughes2023nonisotopicsplittingspheressplit}.
Since we have given a simpler proof in Theorem~\ref{thm:simple-splitting-spheres}, we only give a proof for a weaker version; we leave the rest to the interested reader.

\begin{cor}\label{cor:5d-splitting} Let $\Sigma$ be a standard splitting 3-sphere for the unlink $S^2 \sqcup T^2 \subset S^4$ of a $2$-sphere and a $2$-torus. Then for any $k,\ell \ge 1$, $\Sigma$ and $\boldsymbol{\beta_{v,\ell } \beta_{h,k}}\Sigma $ are non-isotopic rel.\ $S^2 \sqcup T^2$. 
\end{cor}

\begin{proof}
    Let $B, H_h , H_v $ be as in Theorem~\ref{thm:linked-6crit}, where $n=2$ and $S^2 := K_1$.
    We may assume that $B, H_h, H_v$ are disjoint from $\Sigma$ by isotoping $B, H_h, H_v$ rel.\ $\partial$ if necessary.
    Let us denote $\boldsymbol{\beta}:= \boldsymbol{\beta_{v,\ell } \beta_{h,k}} $ for simplicity.
    
    Let us assume that $\Sigma$ and $\boldsymbol{\beta}\Sigma $ are isotopic in $S^4$ rel.\ $S^2 \sqcup T^2 $.
    Recall from Theorem~\ref{thm:linked-6crit} that $(B\sqcup H_{v})\cup\overline{\boldsymbol{\beta}(B\sqcup H_{h})}$ is a nontrivial two-component link of $3$-spheres in $S^5$ such that both components are unknotted.
    We will construct a $5$-ball in $S^5$ such that one component of the link is contained in the interior of this $5$-ball and the other component is contained in the exterior, which contradicts that the link is nontrivial.

    Before we construct the $5$-ball, let us write down explict models of the push-ins of  $B \sqcup H_v$ and $\boldsymbol{\beta}(B \sqcup H_h) $ and hence the link of $3$-spheres.
    We view $S^5 = B^5 \cup S^4 \times [-1,1]\times \overline{B^5}$;
    the link $B \sqcup H_h$ and $\Sigma$ lies in the equatorial $S^4 \times 0\subset S^5$.
    Then, take the following as the push-ins:
    \[((B \sqcup H_v) \times -0.7 ) \cup ((S^2 \sqcup T^2 ) \times [-0.7,0])\ \mathrm{and} \ (\boldsymbol{\beta}(B \sqcup H_h) \times 0.7 ) \cup ((S^2 \sqcup T^2 ) \times [0,0.7]).\]
    Let $B^4 \subset S^4 $ be a $4$-ball that $\Sigma$ bounds. Then, the isotopy from $\Sigma$ to $\boldsymbol{\beta}\Sigma$ rel.\ $S^2 \sqcup T^2 $ must take $B^4$ to $\boldsymbol{\beta} B^4$.
    The 5-ball in $S^5$ is the union of the following:
    \begin{enumerate}
        \item $\boldsymbol{\beta}B^4 \times [0.3,1] \subset S^4 \times [0.3,1]$.
        \item The track of $B^4$ in the isotopy rel.\ $S^2 \sqcup T^2 $ from $B^4$ to $\boldsymbol{\beta}B^4$ in $S^4 \times [-0.3,0.3]$.
        \item $B^4 \times [-1,-0.3] \subset S^4 \times [-1,-0.3]$.\qedhere
    \end{enumerate}
\end{proof}

\appendix

\section{Handle decompositions for Section~\ref{sec:Codimension--submanifolds}}\label{sec:appendixa}

Recall that in Section~\ref{sec:Codimension--submanifolds}, we constructed 3-manifolds $Y\subset S^5$ by viewing $S^5 = B^5 \cup _{S^4} \overline{B^5}$, taking two ``standard'' handlebodies $H_1 , H_2 \subset S^4$ (e.g.\ $H_h$ or $H_v$) with common boundary $K$, constructing an interesting diffeomorphism $\boldsymbol{\beta}$ of $S^4$ rel.\ $N(K)$, pushing $H_1$ into $B^5$ and $\boldsymbol{\beta} H_2$ into $\overline{B^5}$ (which we also denote as $H_1$ and $\boldsymbol{\beta} H_2$), and letting $Y:=H_1 \cup \overline{\boldsymbol{\beta}H_2}$.
We computed $\pi_2 (S^5 \setminus \mathring{N} (Y))$ by studying the belt sphere of the $2$-handles and the attaching spheres of the $3$-handles in an explicit handle decomposition of $S^5 \setminus \mathring{N}(Y)$.

In this appendix, we explain how to obtain a handle decomposition of $S^5 \setminus \mathring{N} (Y)$.
It follows from Lemma~\ref{lem:glue} that $S^5 \setminus \mathring{N} (Y)$ is $B^5 \setminus \mathring{N} (H_1)$ and $\overline{B^5 \setminus \mathring{N} (H_2)}$ glued along $S^4 \setminus \mathring{N} (K)$ via $\boldsymbol{\beta}$. Hence we are left to study the handle decompositions of the exterior of standard handlebodies in $B^5$; this is the content of Proposition~\ref{prop:handle-decomp} (see Figure~\ref{fig:handlebody-appendix}).

\begin{figure}[h]
\begin{centering}
\includegraphics{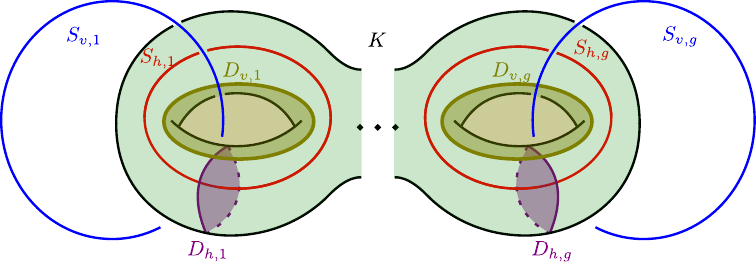}
\par\end{centering}
\caption{\label{fig:handlebody-appendix}The standard genus $g$ surface $K \subset S^3 \times 0 \subset S^4 $, $2$-spheres $S_{h,i},S_{v,i} \subset S^4 \setminus \mathring {N}(K)$ and $2$-disks $D_{h,i}$ (resp.\ $D_{v,i}$) which are compressing disks for the horizontal (resp.\ vertical) handlebody that $K$ bounds in $S^3 \times 0$. The horizontal handlebody is shaded in green.}
\end{figure}

\begin{prop}
\label{prop:handle-decomp}Let $K\subset S^{4}$ be the standard
genus $g$ surface, and let $H\subset B^{5}$ be obtained
by pushing the interior of the horizontal (resp.\ vertical) genus $g$ handlebody in
$S^{4}$ into $B^{5}$.
\begin{enumerate}
\item There exists a handle decomposition of $B^{5}\setminus \mathring{N}(H)$
that consists of one $0$-handle, one $1$-handle, and $g$ many $2$-handles,
such that the belt spheres of the $2$-handles intersected with $S^{4}\setminus \mathring{N}(K)\subset\partial(B^{5}\setminus \mathring{N}(H))$
are the horizontal disks $D_{h,1},\cdots,D_{h,g}$ (resp.\ vertical disks $D_{v,1},\cdots,D_{v,g}$) intersected
with $S^{4}\setminus \mathring{N}(K)$.
\item There exists a relative handle decomposition of $(B^{5}\setminus \mathring{N}(H),S^{4}\setminus \mathring{N}(K))$
that consists of $g$ many $3$-handles, such that the attaching spheres
of the $3$-handles are the vertical spheres $S_{v,1},\cdots,S_{v,g}$ (resp.\ horizontal spheres $S_{h,1},\cdots,S_{h,g}$).
\end{enumerate}
\end{prop}

In \cite[Section~6.2]{MR1707327}, Gompf and Stipsicz explain how to obtain a handle decomposition of such manifolds. 
More generally, let $M$ be an $(n+k-1)$-manifold, and let $\Sigma$ be a properly
embedded $n$-dimensional submanifold of $M\times[-1,1]$. Following \cite[Section~6.2]{MR1707327}, let us recall
how to obtain a relative handle decomposition of the exterior
$(M\times[-1,1])\setminus \mathring{N}(\Sigma)$ relative to $(M\times\{-1\})\setminus \mathring{N}(\Sigma)$, and give an explicit description of the attaching sphere and the belt sphere.
Proposition~\ref{prop:handle-decomp} follows from the below description (Theorem~\ref{thm:62}).

The idea is to work locally: assuming that the height function of
$M\times[-1,1]$ is Morse on $\Sigma$ and considering small neighborhoods
of each index $i$ critical point $p\in\Sigma$ separately, we reduce
to the case where $M$ is the $(n+k-1)$-ball $B$ and $\Sigma_{i}\subset B\times[-1,1]$
is as in Definition~\ref{def:standard-i}. Each index $i$ critical
point gives rise to an $(i+k-1)$-handle; see Theorem~\ref{thm:62}.
\begin{notation}
For $m\ge1$, let $B^{m}:=\{\mathbf{x}\in\mathbb{R}^{m}:\left|\mathbf{x}\right| \le 1\}$
be the closed unit $m$-ball, and let $(a,b)S^{m-1}:=\{\mathbf{x}\in\mathbb{R}^{m}:\left|{\bf x}\right|\in(a,b)\}$.
For $m=0$, let $\mathbb{R}^{0}$ and $B^{0}$ consist of one point
denoted as ${\bf 0}$, and let $(a,b)S^{-1}:=\emptyset$.
\end{notation}

\begin{defn}
\label{def:standard-i}Let $B:=\sqrt{2}B^{n+k-1}\subset\mathbb{R}^{n+k-1}$
and identify $\mathbb{R}^{n+k-1}=\mathbb{R}^{i}\times\mathbb{R}^{n-i}\times\mathbb{R}^{k-1}$.
Then, $\Sigma_{i}\subset B\times[-1,1]$, the \emph{standard properly
embedded $n$-ball with a unique critical point of index $i$}, is
as follows:
\begin{enumerate}
\item For $t\in(0,1]$, $\Sigma_{i}\cap(B\times\{t\})=B^{i}\times S^{n-i-1}\times\{{\bf 0}\}\times\{t\}$.
\item For $t=0$, $\Sigma_{i}\cap(B\times\{t\})=B^{i}\times B^{n-i}\times\{{\bf 0}\}\times\{t\}$.
\item For $t\in[-1,0)$, $\Sigma_{i}\cap(B\times\{t\})=S^{i-1}\times B^{n-i}\times\{{\bf 0}\}\times\{t\}$.
\end{enumerate}
\end{defn}

\begin{thm}[{\cite[Section~6.2]{MR1707327}}]
\label{thm:62}Let $\Sigma_{i}\subset B\times[-1,1]$ be as in Definition~\ref{def:standard-i},
 let $\mathring{N}(\Sigma_{i})$ be an open tubular neighborhood of $\Sigma_{i}$,
and let $\varepsilon>0$ be sufficiently small. Then, $B\times[-1,1]\setminus \mathring{N}(\Sigma_{i})$
can be viewed as 
$(B\times[-1,-1+\varepsilon]\setminus \mathring N(\Sigma_i))\cup(i+k-1)$-handle,
where the attaching sphere is $S_{{\bf xz}}\times\{-1+\varepsilon\}$
where 
\[
S_{{\bf xz}}:=\{(\mathbf{x},\boldsymbol{0},\mathbf{z})\in\mathbb{R}^{i}\times\mathbb{R}^{n-i}\times\mathbb{R}^{k-1}:\left|\mathbf{x}\right|^{2}+\left|\mathbf{z}\right|^{2}=1.1\},
\]
and the belt sphere intersected with $B\times\{1\}$ is $(\{0\}\times B^{n-i}\times\{0\}\times\{1\})\setminus \mathring{N}(\Sigma_{i})$.
\end{thm}

\begin{rem}
We can also describe the core and the cocore of the $(i+k-1)$-handle.
To do this, let us first describe $\mathring{N}(\Sigma_{i})$ explicitly. Let
$\delta:=10^{-10}$ and let $\mathring{N}(\Sigma_{i})$ be such that 
\begin{enumerate}
\item For $t\in[\delta,1]$, $\mathring{N}(\Sigma_{i})\cap(B\times\{t\})=(((1+\delta)\mathring{B}^{i}\times(1-\delta,1+\delta)S^{n-i-1}\times\delta \mathring{B}^{k-1})\cap B)\times\{t\}$.
\item For $t\in(-\delta,\delta)$, $\mathring{N}(\Sigma_{i})\cap(B\times\{t\})=(((1+\delta)\mathring{B}^{i}\times (1+\delta)\mathring{B}^{n-i}\times\delta \mathring{B}^{k-1})\cap B)\times\{t\}$.
\item For $t\in[-1,-\delta]$, $\mathring{N}(\Sigma_{i})\cap(B\times\{t\})=(((1-\delta,1+\delta)S^{i-1}\times (1+\delta) \mathring{B}^{n-i}\times\delta \mathring{B}^{k-1})\cap B)\times\{t\}$.
\end{enumerate}
Then, the core of the $(i+k-1)$-handle is 
\[
S_{{\bf xz}}\times[-1+\varepsilon,2\delta)\cup D_{{\bf xz}}\times\{2\delta\}
\]
where 
\[
D_{{\bf xz}}:=\{(\mathbf{x},\boldsymbol{0},\mathbf{z})\in\mathbb{R}^{i}\times\mathbb{R}^{n-i}\times\mathbb{R}^{k-1}:\left|\mathbf{x}\right|^{2}+\left|\mathbf{z}\right|^{2}\le1.1\},
\]
and the cocore is $\{{\bf 0}\}\times(1-\delta)B^{n-i}\times\{{\bf 0}\}\times[\delta,1]$.
\end{rem}

\begin{figure} 
    \centering
    \includegraphics{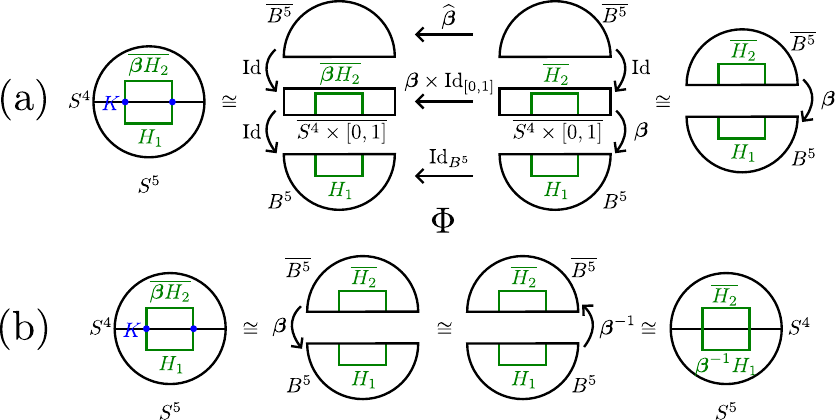}
    \caption{(a): A proof of Equation~(\ref{eq:glue}); here, $\widehat{\boldsymbol{\beta}}$ is a diffeomorphism of $B^5$ that restricts to $\boldsymbol{\beta}$ on the boundary $S^4$. (b): A proof of Equation~(\ref{eq:betainv}).}
    \label{fig:appendix-diffeo}
\end{figure}

Lemma~\ref{lem:glue} lets us study $S^5 \setminus \mathring{N} (Y)$ by studying $B^5 \setminus \mathring{N} (H_1)$, $\overline{B^5 \setminus \mathring{N} (H_2)}$, and $\boldsymbol{\beta}$ separately. 

\begin{lem}\label{lem:glue}
Let $K \subset S^4$ be a (possibly disconnected) surface in $S^4$, let $H_1 , H_2 \subset S^4$ be two (possibly disconnected) handlebodies that $K$ bounds, and let $\boldsymbol{\beta}:(S^4,K) \to (S^4,K)$ be a diffeomorphism of $S^4$ that fixes $K$ setwise. Push the interiors of $H_1$, ${\boldsymbol{\beta}}^{-1} H_1$, $H_2$, and ${\boldsymbol{\beta}} H_2$ into $B^5$, and denote the pushed-in handlebodies as $H_1$, ${\boldsymbol{\beta}}^{-1} H_1$, $H_2$ and ${\boldsymbol{\beta}} H_2$, respectively, by abuse of notation.

Consider the $3$-manifold $H_1 \cup  \overline{{\boldsymbol{\beta}} H_2} \subset S^5$. Then, the pair $(S^5 , H_1 \cup  \overline{{\boldsymbol{\beta}} H_2})$ is diffeomorphic to the pair obtained by gluing $\overline{(B^5 , H_2)}$ to $(B^5 , H_1)$ via ${\boldsymbol{\beta}}$, i.e.
\begin{equation}\label{eq:glue}
(S^5 , H_1 \cup  \overline{ {\boldsymbol{\beta}} H_2}) \cong (B^5 , H_1)\cup _{{\boldsymbol{\beta}} : \partial (B^5 , H_2 ) \to \partial (B^5 , H_1)} \overline{(B^5 , H_2)}.
\end{equation}

We also have
\begin{equation}\label{eq:betainv}
H_1 \cup \overline{{\boldsymbol{\beta}}H_2} \approx {\boldsymbol{\beta}} ^{-1} H_1 \cup  \overline{H_2}.
\end{equation}
\end{lem}

\begin{proof}Figure~\ref{fig:appendix-diffeo}~(a) is the proof of Equation~(\ref{eq:glue}); let us spell this out. By the definition of $(S^5 , H_1 \cup \overline{{\boldsymbol{\beta}} H_2 })$, Equation~(\ref{eq:glue}) is equivalent to 
\begin{equation}\label{eq:gluetmp}
(B^5 , H_1)\cup _{\mathrm{Id} : \partial (B^5 , {\boldsymbol{\beta}} H_2 ) \to \partial (B^5 , H_1)} \overline{(B^5 , {\boldsymbol{\beta}} H_2)} \cong (B^5 , H_1)\cup _{{\boldsymbol{\beta}} : \partial (B^5 , H_2 ) \to \partial (B^5 , H_1)} \overline{(B^5 , H_2)}.
\end{equation}

Let $N\cong S^4\times [0,1]$ be a closed collar neighborhood of the boundary $S^4 = \partial B^5$, which is identified with $S^4 \times 0 \subset N$. 
Without loss of generality, assume that the pushed-in handlebodies $H_1, H_2, {\boldsymbol{\beta}} H_2 \subset B^5$ are contained in $N$. Let $B := B^5 \setminus \mathring N$: then we have $$B^5 = N\cup B \cong (S^4 \times [0,1])\cup B.$$

Hence, we can rewrite Equation~(\ref{eq:gluetmp}) as 
\begin{align}
(B^5 , H_1)\cup _{\mathrm{Id} : (S^4 \times 0 , \partial ({\boldsymbol{\beta}} H_2 )) \to \partial (B^5 , H_1)} \overline{(S^4 \times [0,1], {\boldsymbol{\beta}} H_2)} \cup_{\mathrm{Id} :  (\partial B , \emptyset) \to (S^4 \times 1 , \emptyset ) } \overline{(B , \emptyset )} \label{eq:gluetmp10}\\
\cong (B^5 , H_1)\cup _{{\boldsymbol{\beta}} : (S^4 \times 0 , \partial  H_2 ) \to \partial (B^5 , H_1)} \overline{(S^4 \times [0,1],  H_2)} \cup_{\mathrm{Id} :  (\partial B , \emptyset) \to (S^4 \times 1 , \emptyset ) } \overline{(B , \emptyset )}.\label{eq:gluetmp11}
\end{align}

Let us define a diffeomorphism $\Phi$ from (\ref{eq:gluetmp11}) to (\ref{eq:gluetmp10}). Note that both (\ref{eq:gluetmp10}) and (\ref{eq:gluetmp11}) have three summands; $\Phi$ preserves this decomposition.
(1) On the first summand, let $$\Phi := \mathrm{Id}_{B^5} : \overline{(B^5 , H_1 )}\to \overline{(B^5 , H_1 )}.$$
(2) On the second summand, let $$\Phi:= {\boldsymbol{\beta}} \times \mathrm{Id}_{[0,1]} : \overline{(S^4 \times [0,1],  H_2)} \to \overline{(S^4 \times [0,1], {\boldsymbol{\beta}} H_2)}.$$
(3) Since $S^5$ has a unique smooth structure \cite{MR148075}, any diffeomorphism of $S^4$ is pseudoisotopic to the identity (see e.g.\ \cite[Lemma~12]{MR4996240}), and so there exists some diffeomorphism $\widehat{{\boldsymbol{\beta}}} :B \to B$ that restricts to ${\boldsymbol{\beta}}$ on its boundary. On the third summand, let $\Phi := \widehat{{\boldsymbol{\beta}}}$.

Figure~\ref{fig:appendix-diffeo}~(b) is the proof of Equation~(\ref{eq:betainv}); let us spell this out. 
Since\footnote{In fact, let $f,g:M\to S^n$ be two embeddings of a manifold $M$ such that their images are disjoint from some $n$-ball in $S^n$. Then, $f$ and $g$ are isotopic if and only if there exists a diffeomorphism $\varphi :S^n \to S^n$ such that $f=\varphi \circ g$.
The only if direction follows from the isotopy extension theorem.
To show the if direction, let $B^n \subset S^n $ be an $n$-ball disjoint from $f(M)$; we can find a diffeomorphism $\varphi '$ of $S^n$ that fixes $B^n$ pointwise and is isotopic to $\varphi $.
Let $(\psi _t)_{t\in [0,1]}$ be an isotopy of $S^n$ such that $\psi_0 = {\mathrm {Id}}$ and $\psi_1 (S^n \setminus B^n )\subset B^n$.
Then, $$g \approx \varphi ^{-1} \circ  \psi_1 \circ \varphi  \circ g  = \varphi ^{-1} \circ \psi _1 \circ f \approx {\varphi '}^{-1} \circ \psi _1 \circ  f = f,$$
where the last equality holds since $\psi_1 (f(M) )\subset \psi_1 (S^n \setminus B^n )\subset B^n$ and $\varphi '$ fixes $B^n$ pointwise.} 
two $3$-knots in $S^5$ are isotopic if and only if they are related by a diffeomorphism of $S^5$, Equation~(\ref{eq:betainv}) is equivalent~to
\begin{equation}\label{eq:betainv2}
(S^5 , H_1 \cup \overline{{\boldsymbol{\beta}}H_2}) \cong (S^5 , {\boldsymbol{\beta}} ^{-1} H_1 \cup  \overline{H_2}).
\end{equation}
Similarly to the above proof of Equation~(\ref{eq:glue}), we can show that 
\begin{equation}\label{eq:gluetmp2}
(B^5 ,{\boldsymbol{\beta}}^{-1} H_1)\cup _{\mathrm{Id} : \partial (B^5 , H_1 ) \to \partial (B^5 , H_2)} \overline{(B^5 ,  H_2)} \cong (B^5 , H_1)\cup _{{\boldsymbol{\beta}} ^{-1} : \partial (B^5 , H_1 ) \to \partial (B^5 , H_2)} \overline{(B^5 , H_2)},
\end{equation}
and so Equation~(\ref{eq:betainv2}) follows from Equations~(\ref{eq:glue})~and~(\ref{eq:gluetmp2}), together with the formal statement
$$(B^5 , H_1)\cup _{{\boldsymbol{\beta}} : \partial (B^5 , H_2 ) \to \partial (B^5 , H_1)} \overline{(B^5 , H_2)} \cong (B^5 , H_1)\cup _{{\boldsymbol{\beta}} ^{-1} : \partial (B^5 , H_1 ) \to \partial (B^5 , H_2)} \overline{(B^5 , H_2)}. \qedhere$$
\end{proof}

\section{$(2n-1)$-knots in the $(2n+1)$-sphere with 4 critical points}\label{sec:kuiperappendix}

In this appendix we generalize the construction of Theorem~\ref{thm:morsesimple-s3} to construct $(2n-1)$-knots in $S^{2n+1}$ with one critical point of each index $0, n-1, n$, and $2n-1$.
As mentioned in the introduction, this answers a question of Kuiper \cite[Section 10, page 390]{kuiper1984geometryintotal}, on whether there exist nontrivial $(4m-1)$-knots in $S^{4m+1}$ with 4 critical points with respect to the standard height function on $S^{4m+1}$ for $m \geq 1$.

\begin{thm}\label{thm:higher-dim-knots}
There exist infinitely many pairwise non-isotopic
embeddings of $S^{2n-1}$ in $S^{2n+1}$ all of which have four critical points with respect
to the standard height function on $S^{2n+1}$ (which is Morse on the $(2n-1)$-knot).
\end{thm}

\begin{proof}The proof is analogous to the proof of Theorem~\ref{thm:morsesimple-s3}, but we instead use $2n$-dimensional barbell diffeomorphisms.
In $2n$ dimensions, Budney and Gabai \cite[Section~2]{bghyperbolic} construct a barbell map $\beta :\mathcal{NB} \to \mathcal{NB}$ on the \emph{model thickened barbell} $\mathcal{NB}:=S^n \times B^n \natural S^n \times B^n $ analogously to the $4$-dimensional case. We will push forward $\beta$ along some embeddings of $\mathcal{NB}$. The analogous statement to Lemma~\ref{lem:barbelltubing} is \cite[Proposition~2.9]{bghyperbolic}.

View $S^{2n}=(\mathbb{R}^{n}\times\mathbb{R}^{n})\cup\{\infty\}$,
and let $T:=S^{n-1}\times S^{n-1}$, $H_{v}:=S^{n-1}\times B^{n}$,
and $H_{h}:=B^{n}\times S^{n-1}$ all in $\mathbb{R}^{n}\times\mathbb{R}^{n}\subset S^{2n}$.
Let $S_{h}:=(\{\mathbf{0}\}\times\mathbb{R}^{n})\cup\{\infty\}$,
$S_{v}:=(\mathbb{R}^{n}\times\{\mathbf{0}\})\cup\{\infty\}$, and
$D_{v}:=\{x\}\times B^{n}$ for some $x\in S^{n-1}$.
Let $k,\ell\ge1$ and let $\beta_{h,k}$ (resp.\ $\beta_{v,\ell}$)
be a barbell in $S^{2n}\setminus N(T)$ such that its two cuffs are
parallel copies of $S_{h}$ (resp.\ $S_{v}$) and its bar winds $k$
(resp.\ $\ell$) times around the meridian of $T$. Push the interiors of $\boldsymbol{\beta}_{v,\ell}\boldsymbol{\beta}_{h,k}H_{h}$ and $H_{v}$ into $B^{2n+1}$ and denote
them as $\boldsymbol{\beta}_{v,\ell}\boldsymbol{\beta}_{h,k}H_{h}$ and $H_{v}$ as well. Let $Y:=H_{v}\cup\overline{\boldsymbol{\beta}_{v,\ell}\boldsymbol{\beta}_{h,k}H_{h}}\subset S^{2n+1} = B^{2n+1} \cup \overline{B^{2n+1}}$.
We will show that 
$\dim_{\mathbb{F}_{2}} H_{n}(\widetilde{S^{2n+1}\setminus Y};\mathbb{F}_2 ) =2k+2\ell+2$,
where $\widetilde{S^{2n+1}\setminus Y}$ denotes the universal cover.

First, by Theorem~\ref{thm:62}, $B^{2n+1}\setminus\mathring{N}(H_{v})$
has a handle decomposition with one $0$-, $1$-, and $n$-handle
each, where the belt sphere of the $n$-handle restricted to $S^{2n}\setminus\mathring{N}(T)$
is $D_{v}\setminus\mathring{N}(T)$. Also by Theorem~\ref{thm:62}, $(B^{2n+1}\setminus\mathring{N}(H_{h}), S^{2n} \setminus \mathring{N}(T))$
has a relative handle decomposition with one $(n+1)$-handle, and the attaching
sphere of the $(n+1)$-handle is $S_{v}$.\footnote{For the readers' convenience, we record the following intermediate
description of the attaching sphere that Theorem~\ref{thm:62} provides: let
$\ell_{x}$ be the straight line segment between $0.9x$ and $1.1x$.
Then the attaching sphere is isotopic to $(1.1S^{n-1}\times\ell_{x})\cup(1.1B^{n}\times\partial\ell_{x})$,
which is isotopic to $S_{v}$.} Hence, by Equation~(\ref{eq:glue}) for $\boldsymbol{\beta}:=\boldsymbol{\beta}_{v,\ell}\boldsymbol{\beta}_{h,k}$,\footnote{In the proof of Lemma~\ref{lem:glue}, we used that any diffeomorphism of $S^{4}$
is pseudoisotopic to the identity. Note that this is not true for $S^{2n}$. However, Equation~(\ref{eq:glue}) still holds since our specific $\boldsymbol{\beta}$
is pseudoisotopic to the identity.
In fact, $\boldsymbol{\beta}_{h,k}$ and $\boldsymbol{\beta}_{v,\ell}$
are isotopic to the identity in $S^{2n}$, and so $\boldsymbol{\beta}$
is also isotopic to the identity. Indeed, $\boldsymbol{\beta}_{h,k}$
is isotopic to the identity since a cuff of $\beta_{h,k}$ bounds
a $B^{n+1}$ whose interior is disjoint from $\beta_{h,k}$. Similarly,
$\boldsymbol{\beta}_{v,\ell}$ is isotopic to the identity.} $S^{2n+1}\setminus\mathring{N}(Y)$ has a handle decomposition with
one $0$-, $1$-, $n$-, and $(n+1)$-handle each, where the belt sphere
of the $n$-handle restricted to $S^{2n}\setminus\mathring{N}(T)$
is $D_{v}\setminus\mathring{N}(T)$, and the attaching sphere of the
$(n+1)$-handle is $\boldsymbol{\beta}_{v,\ell}\boldsymbol{\beta}_{h,k}S_{v}$.

Now, $H_{n}(\widetilde{S^{2n+1}\setminus\mathring{N}(Y)};\mathbb{F}_{2})=\mathbb{F}_{2}[t,t^{-1}]/(f)$
where $f$ is the equivariant intersection number modulo
$2$ of $D_{v}$ and $\boldsymbol{\beta}_{v,\ell}\boldsymbol{\beta}_{h,k}S_{v}$.
The computation of $f$ is the same as in the
proof of Theorem~\ref{thm:morsesimple-s3}, and we have $f=1+(t+t^{-1})(t^{k}+t^{-k})(t^{\ell}+t^{-\ell})$:
first, the equivariant intersection number modulo~$2$ of $S_{h}$
and $S_{v}$ is $1+t\in\mathbb{F}_{2}[t,t^{-1}]$, and so the computation
of the homology class of a lift of $\boldsymbol{\beta}_{v,\ell}\boldsymbol{\beta}_{h,k}S_{v}$
to the universal cover of $S^{2n}\setminus\mathring{N}(T)$ is as in Equation~(\ref{eq:barbell-sphere-homology-class}). Second, since $D_{v}$ and $S_{v}$ meet transversely at
one point and $D_{v}$ and $S_{h}$ are disjoint, the computation of $f$ is also the same, and we have $f=1+(t+t^{-1})(t^{k}+t^{-k})(t^{\ell}+t^{-\ell})$.
\end{proof}

\bibliographystyle{amsalpha}
\bibliography{bib}

\end{document}